\documentclass{preprint}

\usepackage{amssymb,amsfonts,amsmath, epsf,psfrag, graphicx,cancel,xcolor}
\usepackage{graphics,epsfig}
\usepackage{latexsym}
\usepackage{color}
\usepackage{verbatim}

\newcommand{\ds}{\displaystyle }

\newcommand{\BI}{\boldsymbol{\mathcal{I}}}

\newcommand{\BIT}{\boldsymbol{\mathcal{I}^n_{tot}}}

\newcommand{\BITp}{\boldsymbol{\mathcal{I}^{n+1}_{tot}}}

\newtheorem{proposition}{Proposition}

\numberwithin{equation}{section}

\begin{document}

\title[Hyperbolic chemotaxis on a network]{A hyperbolic model of chemotaxis on a network: a numerical study}

 \author{G. Bretti}
 \address{Istituto per le Applicazioni del Calcolo ``M. Picone" -- Consiglio Nazionale delle Ricerche, Via dei Taurini 19,  Rome, Italy ;  \email{g.bretti@iac.cnr.it  \ \& \   roberto.natalini@cnr.it}}
  
 \author{R. Natalini}
  \sameaddress{1}
 
 \author{M. Ribot}
\address{ Laboratoire J.A.Dieudonn\'e, UMR 6621 CNRS, Universit\'e de Nice Sophia Antipolis, Nice, France  ; \email{ribot@unice.fr}}
\secondaddress{Project Team COFFEE, INRIA Sophia-Antipolis, France}

\date{}

\begin{abstract}
In this paper we deal with a semilinear hyperbolic  chemotaxis model in one space dimension evolving on a network, with suitable transmission conditions at nodes. This framework is motivated by tissue-engineering scaffolds used for improving wound healing. We introduce a numerical scheme, which guarantees  global mass densities  conservation. Moreover our scheme is able to yield a correct approximation of the effects of the source term at equilibrium.   Several numerical tests are presented to show the behavior of solutions and to discuss the stability and the accuracy of our approximation.   

\end{abstract}


\maketitle

\section{Introduction}\label{intro}

The movement of bacteria, cells or  other microorganisms under the effect of a chemical stimulus, represented by a chemoattractant, has been widely studied in mathematics in the last two decades, see \cite{H, M, Pe}, and numerous models involving partial differential equations have been proposed. 
The basic unknowns in these chemotactic models are the density of individuals and the concentrations of some chemical attractants. One of the most considered models is the Patlak-Keller-Segel system \cite{KS}, where the evolution of the density of cells is described by a parabolic equation, and the concentration of  a chemoattractant is generally given by a parabolic or elliptic equation, depending on the different regimes to be described and on authors' choices. The behavior of this system is quite well known now: in the one-dimensional case, the solution is always global in time, while in two and more dimensions the solutions exist globally in time
or blow up according to the size of the initial data. However, a drawback of this model is that the diffusion leads to a fast dissipation or an explosive behavior,  and prevents us to observe intermediate organized structures, like aggregation patterns. 

By contrast, models based on hyperbolic/kinetic equations  for the evolution of the density of individuals, are characterized by a finite speed of propagation and have registered a growing consideration in the last few years   \cite{GambaEtAl03,Pe, FilbetLaurencotPerthame, Hillen,DolakHillen03}. In such models, the population is divided in compartments depending on the
velocity of propagation of individuals, giving raise to kinetic type equations, either with continuous or discrete velocities. 

Here we consider an hyperbolic-parabolic system which arises as a simple model for chemotaxis:
\begin{equation}\label{hyper-gen}
 \left\{\begin{array}{ll}
  u_t +v_x =0,\\ 
  v_t +\lambda^2 u_x 
   =\phi_x\,u-v,\\ 
   \phi_{t}-D\, \phi_{xx}=au-b\phi.
 \end{array}\right.
\end{equation}
 Such kind of models were originally considered in \cite{Segel77},
and later reconsidered in \cite{GreenbergAlt87}. They are based on an adaptation to the chemotactic case of the so-called hyperbolic heat or Cattaneo or telegraph equation, adding a source term accounting for the chemotactic motion  in the equation for the flux. The function $u$ is the density of cells in the considered medium, $v$ is their averaged flux and $\phi$ denotes the density of chemoattractant. The individuals move at a constant speed $\lambda$, changing their direction along the axis during the time. The positive constant $D$ is the diffusion coefficient of the chemoattractant; the positive coefficients $a$ and $b$, are respectively its production and degradation rates.  

These equations are expected to behave asymptotically as the corresponding parabolic equations, but displaying a different and richer transitory regime, and this is what is known to happen at least without the chemotactic term. 
Analytically,   these models  have been studied   in
\cite{HS, HillenRohdeLutscher01} and more recently in \cite{GMNR09}, where the analytical features were almost completely worked out, at least around constant equilibrium states, where it is proved that, at least for the Cauchy problem, the solutions of the hyperbolic and parabolic models are close for large times. 

The novelty of this paper is to consider this one dimensional model on  a network. More precisely, we consider system in the form \eqref{hyper-gen}  on each arc of the network, and so we have to consider  one set of solutions $(u,v,\phi)$ for each arc. Functions on different arcs are  coupled  using suitable transmission conditions  on each node of the network. Conservation laws or wave equations on networks  have  already been studied, for example in \cite{GaravelloPiccoli} for traffic flows or in \cite{DagerZuazua, ValeinZuazua} for flexible strings distributed along a planar graph.
However, here  we consider  different types of transmission conditions, which impose the  continuity of the fluxes rather than the continuity of the densities. Therefore, in this article,  a particular care will be given to the proper setting and the numerical approximation of the transmission conditions at  nodes, both for the hyperbolic and the parabolic parts of  \eqref{hyper-gen}. In particular, some conditions have to be imposed on the approximation of the boundary conditions, in order to ensure the conservation of the total mass of the system. Let us also mention that a first analytical study of system \eqref{hyper-gen} on a network, coupled through transmission conditions of this type, is carried out in \cite{GPhD}.

The study of this system is motivated by the tissue-engineering  research concerning the  movement of fibroblasts on  artificial scaffolds \cite{Harley08, mandal09, Spadaccio}, during the process of dermal wound healing. The natural process of healing of a damaged tissue occurs through a first phase in which fibroblasts, the  stem cells to be in charge of to the reparation of dermal tissue,  create a new extracellular matrix, essentially made by collagen, and, driven by chemotaxis, migrate to fill the wound. In recent years, tissue-engineering research has developed some new techniques, which  aim at accelerating the wound healing. Actually,  cellular migration on an injured body zone is stimulated and improved by using artificial scaffolds constituted by a network of crossed polymeric threads inserted within the wound, which mimic the extracellular matrix. The fibroblasts's reparation action is accelerated, since they already have a support and also they are constrained to move along the network, with less degrees of freedom, and it is believed that this approach could be effective in minimizing scarring \cite{Spadaccio}. Therefore, our simple model of chemotaxis on a network, which can be obtained by reducing the kinetic model of cell movement on a 3D extracellular matrix proposed in  \cite{CP} to the case of a network, see \cite{GPhD}, is a good candidate for reproducing this configuration: the arcs of the network  stand for the fibers of the scaffold and the transport equations give the evolution of the density of fibroblasts on each fiber. However, in this paper, we only address the numerical aspects of this problem. More direct applications of this framework to the real biomedical problem will be explored in  future research. As reported in \cite{Harley08}, the present understanding of the critical biochemical and biophysical parameters that affect cell motility in three-dimensional environments is quite limited. Nevertheless, it has been observed that junction interactions affect
local directional persistence as well as cell speed at and away from the junctions, so providing a new  mechanism to control cell motility by using the extracellular microstructure. Therefore, mathematical modeling and simulations could play a crucial role in providing a better understanding of these phenomena, and an optimization tool for designing improved scaffolds.

The main focus of this paper is on the construction of  an effective numerical scheme for computing the solutions to this problem, which is not an easy task, even for the case of a single arc. In that case,  non constant highly concentrated stationary solutions  are expected and schemes which are able to  capture these large gradients in an accurate way are needed. The main problem is  to balance correctly  the source term with the differential part, in order to avoid  an incorrect approximation of the density flux at equilibrium, as first observed in \cite{GMNR09}. {\em Asymptotic High Order schemes} (AHO) were introduced in \cite{NaRi}, inspired by  \cite{ABN}, to deal with this kind of inaccuracies. These schemes  are based on standard finite differences  methods, modified by a suitable treatment of the source terms, and they take into account for the behavior of the solutions near non constant stationary states. An alternative  approach, inspired by the well-balanced methods, has been proposed in \cite{Gosse10,Gosse11}, with similar results. However the methods in \cite{NaRi}  seem easier to be generalized to the present framework.

Regarding the problem considered in this paper, the main difficulty is in the discretization of the transmission conditions at  node, also enforcing global  mass conservation at the discrete level.  Therefore, in Section \ref{sec2} we explain some analytical properties of problem \eqref{hyper-gen}, with a particular emphasis on boundary and transmission conditions. Section \ref{sec3} is devoted to the numerical approximation of the problem based on a AHO  scheme with a suitable   discretization of the transmission and  boundary conditions ensuring the mass conservation.  In the present paper, we have chosen to consider only the second order version of the scheme, which is enough for our purposes, but it is easy to adapt also  the third order schemes proposed in \cite{NaRi}. 
Remark that here, unlike the single interval case,  we are forced,  for any given time step, to fix the space step on each arc using relation \eqref{cfl} introduced in Section 3, to obtain consistency on the boundary. Numerical tests (not shown) confirm the necessity of this supplementary constraint.

Finally, in Section \ref{sec4}, we report some numerical experiments, to show the behavior and the stability of our scheme. A special attention is given to the stability of the scheme near nodes and the correct behavior of the approximation for large times and near asymptotic states. It has to be mentioned that during this research we observed, in contrast with what happens for the diffusive models, the appearance of blow-up phenomena even for data of relative moderated size. Even if, up to now, there are no rigorous results, which can help to decide if these singular events are really occurring, or they are just a numerical artifact, our close investigation in Subsection 4.3 gives a strong indication towards the first alternative. 

\section{Analytical background}\label{sec2}
Let us define a network or a connected  graph $G=(\mathcal{N},\mathcal{A})$, as  composed of two finite sets,  a set of $P$ nodes (or  vertices) $\mathcal{N}$  and  a set of $N$ arcs (or edges) $\mathcal{A}$, such that an arc connects a pair of nodes. Since arcs are bidirectional the graph is non-oriented, but we need to fix an artificial orientation in order to fix a sign to the velocities. 
 The network is therefore  composed of "oriented" arcs and there are  two different types of intervals at a node  $p \in \mathcal{N}$ : incoming ones -- the set of these intervals is denoted by $I_{p}$ -- and outgoing ones  -- whose set  is denoted by $O_{p}$. For example, on the network depicted in Figure \ref{fig:4a1n}, $1,2 \in I$ and $3,4 \in O$.
 We will also denote in the following by $I_{out}$ and $O_{out}$ the set of the arcs incoming or outgoing from the outer boundaries.
 The $N$ arcs of the network are parametrized as intervals $a_i=[0,L_i]$, $i=1,\ldots,N$, and for an incoming arc, $L_{i}$ is the abscissa of the node, whereas it is $0$ for an outgoing arc.

\begin{figure}[htbp!]
\begin{center}
\includegraphics[height=5cm,width=5cm]{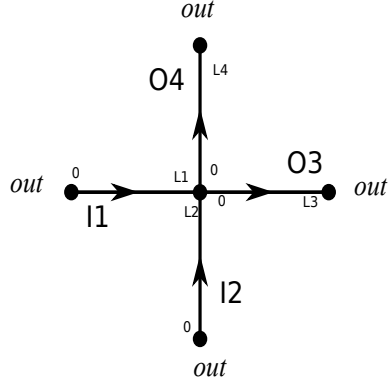}
\caption{First example of network.}
\label{fig:4a1n}
\end{center}
\end{figure}

\subsection{Evolution equations for  the problem}

We consider system (\ref{hyper-gen}) on each arc and  rewrite it in diagonal variables for its hyperbolic part by setting
\begin{equation}
\label{cvar}
  u^{\pm}=\frac{1}{2}\left(u\pm \frac{v}{\lambda}\right).
  \end{equation} 
Here $\ds u^{+}$ and  $\ds u^{-}$ are the Riemann invariants of the system and $\ds u^{+}$ (resp. $\ds u^{-}$)  denotes the density of cells following the orientation of the arc (resp. the density of cells going in the opposite direction). This transformation is inverted by $ u=u^++u^-$  and  $v=\lambda(u^+-u^-)$, and yields:
\begin{equation}
\label{hyper-gen-diag}
\left\{
\begin{aligned}
u^+_{t}+\lambda u^+_{x}&=\frac{1}{2\lambda}\left((\phi_{x}-\lambda)u^++(\phi_{x}+\lambda)u^-\right),\\
u^-_{t}-\lambda u^-_{x}&=-\frac{1}{2\lambda}\left((\phi_{x}-\lambda)u^++(\phi_{x}+\lambda)u^- \right),\\
\phi_{t}-D\phi_{xx}&= a(u^++u^-)-b\phi.
\end{aligned}
\right.
\end{equation}
 We can also denote by $T^{\pm}=\frac{1}{2\lambda}(\phi_x\mp\lambda)$ the turning rates (namely the probabilities of cells to change direction) and $a(u^++u^-)-b\phi$ represents the production and degradation  of the chemoattractant.
We assume that all the cells are moving along an arc with the same velocity $\lambda$ (in modulus), which may depend however on the characteristics of the arc. For the moment, we omitted the indexes related to the arc number since no confusion was possible. From now on, however, we need to distinguish the quantities on different arcs and we denote by $u_{i}^{\pm}$, $u_{i}$ , $v_{i}$ and $\phi_{i}$ the values of the corresponding variables on the $i$-th arc.
On the outer boundaries, we could consider general  boundary conditions:
\begin{equation}\label{boundary}
\left\{
\begin{aligned}
u_{i}^+(0,t)=\alpha_{i}(t)u_{i}^-(0,t)+\beta_{i}(t) ,&  \textrm{ if } i\in I_{out}, \\
u_{i}^-(L_{i},t)=\alpha_{i}(t)u_{i}^+(L_{i},t)+\beta_{i}(t), & \textrm{ if } i\in O_{out}.
\end{aligned}
\right.
\end{equation}
For $\alpha_{i}(t)=1$ and $\beta_{i}(t)=0$, we just recover the standard no-flux boundary condition
\begin{equation}\label{Neumann}
 u_{i}^+(.,t)=u_{i}^-(.,t) \, (\textrm{which is equivalent to } v(.,t)=0).
\end{equation}  

On the outer boundaries, we also consider no-flux (Neumann) boundary conditions for $\phi$, which read  
\begin{equation}\label{Neumann-phi}
\partial_{x}\phi_{i}(.,t)=0.
\end{equation}

The no-flux boundary conditions mean that, on the boundary, the fluxes of cells and chemoattractants  are null. This condition could be generalized, for example in the case when we assume that there is a production of fibroblasts on the boundary.

\subsection{Transmission conditions at a node}

Now, let us describe how to define the conditions at a node; this is an important point, since the behavior of the solution will be very different according to the conditions we choose. Moreover, let us recall that the coupling between the densities on the arcs are obtained through these conditions. 
At   node $p \in \mathcal{N}$, we have to give values to the components such that the corresponding characteristics are going out of the node. Therefore, we consider the following transmission conditions at node:
\begin{equation}\label{transmission}
\left\{
\begin{aligned}
u_{i}^-(L_{i},t)=\sum_{j \in I_{p}} \xi_{i,j} u_{j}^+(L_{j},t)+\sum_{j \in O_{p}} \xi_{i,j} u_{j}^-(0,t) ,&  \textrm{ if } i\in I_{p}, \\
u_{i}^+(0,t)=\sum_{j \in I_{p}} \xi_{i,j} u_{j}^+(L_{j},t)+\sum_{j \in O_{p}} \xi_{i,j} u_{j}^-(0,t), & \textrm{ if } i\in O_{p},
\end{aligned}
\right.
\end{equation}
where the constant $\xi_{i,j} \in [0,1]$ are the transmission coefficients: they represent the probability that a cell at a node decides to move from the $i-$th to the $j-$th arc of the network, also including the turnabout on the same arc. 
Let us notice that the condition differs when the arc is an incoming or an outgoing arc. Indeed, for an incoming (resp. outgoing) arc, the value of the function $u_{i}^+$ (resp. $u_{i}^-$)  at the node is obtained through the system and we need only to define $u_{i}^-$ (resp. $u_{i}^+$) at the boundary. 

 These  transmission conditions do not guarantee the continuity of the densities at node; however, we are interested in having the continuity of the fluxes at the node, meaning that we cannot loose nor gain any cells during the passage through a node. This is obtained using a condition mixing the transmission coefficients  $\xi_{i,j}$ and the velocities of the  arcs connected at node $p$.
Fixing a node and denoting the velocities of the  arcs by  $\lambda_{i}, \,  i \in I_p \cup O_{p}$, in order to have the flux conservation at  node $p$, which is given by: 
\begin{equation}\label{flux}
 \sum_{i \in I_{p} } \lambda_{i} (u_{i}^+(L_{i},t)-u_{i}^-(L_{i},t))=\sum_{i \in O_{p} } \lambda_{i} (u_{i}^+(0,t)-u_{i}^-(0,t)), 
\end{equation}
it is enough to impose the following conditions:  
\begin{equation}\label{condition_lambda}
\sum_{i \in I_p \cup O_{p}} \lambda_{i}  \xi_{i,j} =\lambda_{j},  \, j \in I_p \cup O_{p}.
\end{equation}

Notice that, condition \eqref{flux}, can be rewritten  in the $u-v$ variables as 
 \begin{equation}\label{flux2}
  \sum_{i \in I_{p} } v_{i}(L_{i},t)=\sum_{i \in O_{p} } v_{i}(0,t).
 \end{equation}
This condition  ensures that  the global mass $\mu(t)$  of the system is conserved along the time, namely:
\begin{equation}\label{masscons}
\mu(t) = \sum_{i=1}^N \int_{0}^{L_{i}} u_{i}(x,t) dx = \mu_0 := \sum_{i=1}^N \int_{0}^{L_{i}} u_{i}(x,0) dx, \, \textrm{ for all } t>0.
\end{equation}

\subsection{Dissipative transmission coefficients for the hyperbolic problem.}\label{SS-dissip-trans}
 
It is sometimes useful to restrict  our attention to the case of
positive transmission coefficients of dissipative type, in the sense that they ensure energy decay of the solutions to  the linear version of system (\ref{hyper-gen}), namely:
\begin{equation}\label{hyper-lin}
 \left\{\begin{array}{ll}
  u_t +v_x =0,\\ 
  v_t +\lambda^2 u_x 
   =-v,
 \end{array}\right.
\end{equation}
on a general network, with no-flux conditions \eqref{Neumann} on the external nodes, and transmission conditions \eqref{transmission} at the internal nodes, always assuming the flux conservation condition \eqref{condition_lambda} at nodes.
 
To obtain the decay in time of the energy, which is defined by
\begin{equation*}
E(t)=  \left(\sum_{i=1}^N \int_{0}^{L_{i}} \left(u_{i}^2(x,t)+\frac{v_{i}^2(x,t)}{\lambda_{i}^2}\right) dx \right)^{1/2},
\end{equation*} 
it is sufficient to impose some equalities on the coefficients, as proved in \cite{GPhD}.

\begin{proposition}[\cite{GPhD}]\label{prop-dissip}
The energy associated with the solutions to system \eqref{hyper-lin}, with no-flux conditions \eqref{Neumann} on the external nodes, and transmission conditions \eqref{transmission} at the internal nodes, assuming condition \eqref{condition_lambda}, is 
decreasing if the transmission coefficients $\xi_{i,j}$ belong to $[0,1]$, and at  every node $p \in \mathcal{N}$, we have:
\begin{equation}\label{dissip-coeff}
 \sum_{j  \in I_p \cup O_{p}}\xi_{i,j} = 1\textrm{ for all } i  \in I_p \cup O_{p}.
\end{equation}
\end{proposition}

Actually, in  \cite{GPhD}, it is proved that under the assumptions of  Proposition \ref{prop-dissip},  it is possible to define a monotone generator of semigroup, and then a contraction semigroup, in the Sobolev space $H^1$, for the linear transmission problem \eqref{hyper-lin} on a network. Let us remark also that in the simplest case of a network composed by two arcs (one incoming and one outgoing, see next   Figure \ref{fig:rete}), these conditions are also necessary  in order to have the dissipation property. In such a case we have that dissipativity is given iff: 
\begin{equation}\label{dissip-coeff2}
\max\left\{0,\frac{\lambda_1-\lambda_2}{\lambda_1}\right\} \le \xi_{1,1} \le 1,\  \lambda_2 (1-\xi_{2,2}) = \lambda_1(1-\xi_{1,1}) .
\end{equation}
Using the previous relations and conditions on the coefficients $\xi_{i,j}$ given by \eqref{condition_lambda}, we obtain the values for the two missing coefficients:
\begin{equation}\label{dissip-coeff22}
\xi_{1,2} = 1 - \xi_{1,1}, \ 
 \xi_{2,1}= \frac{\lambda_1}{\lambda_2}(1 - \xi_{1,1}),
\end{equation}
so, we have only one degree of freedom. 

\subsection{Transmission conditions for $\phi$}
    
Now let us consider the transmission conditions for $\phi$ in  system \eqref{hyper-gen}. We complement conditions \eqref{boundary}, \eqref{Neumann-phi}, and \eqref{transmission} with   a transmission condition  for $\phi$. 
As previously, we do not impose the continuity of the  density of chemoattractant  $\phi$, but only  the continuity of the flux at  node $p \in \mathcal{N}$. Therefore, we use  the Kedem-Katchalsky permeability condition \cite{KK}, which has been first proposed in the case of flux through a membrane. For some positive coefficients $\kappa_{i,j}$, we impose at node
\begin{equation}
D_{i}\partial_{n}\phi_{i}=\sum_{j \in I_{p} {\cup} O_{p}} \kappa_{i,j} (\phi_{j}-\phi_{i}), \  i\in I_{p}\cup O_{p}.
\label{transmission-phi}
\end{equation}
 The condition
\begin{equation}
\kappa_{i,j}= \kappa_{j,i}, i,j=1,\ldots,N
\label{symetrie-kappa}
\end{equation}
yields the  conservation of the fluxes at  node $p$, that is to say 
\[\ds \sum_{i \in I_{p} {\cup} O_{p}} D_{i}\partial_{n}\phi_{i}=0. \]
Let us also notice that we can assume that $\ds \kappa_{i,i}=0, \, i=1,\ldots,N$, which does not change condition \eqref{transmission-phi}. Finally, notice that the positivity of the transmission coefficients  $\kappa_{i,j}$, guarantees the energy dissipation for the equation for $\phi$ in  \eqref{hyper-gen}, when the term in $u$ is absent. 

\subsection{Stationary solutions}\label{SS-staz}
First  we consider stationary solutions, which are known to drive the asymptotic behavior of the system.
Let us consider the case of stationary solutions of  system \eqref{hyper-gen}, complemented with boundary conditions \eqref{Neumann}, \eqref{Neumann-phi}, \eqref{transmission}, and \eqref{transmission-phi}. In the general case, we find on each arc the following solution~:
\begin{equation}\label{stat}
 \left\{\begin{array}{l}
  v_{i}=constant,\\ 
  u_{i}=\exp(\phi_{i}/\lambda_{i}^2)\left( C_{i}-\ds\frac{v_{i}}{\lambda_{i}^2}\int_{0}^x \exp(-\phi_{i}(y)/\lambda_{i}^2) dy \right), \\
  -D_{i}\phi_{i,xx}=a_{i}u_{i}-b_{i}\phi_{i},
 \end{array}\right.
\end{equation}
which leads to solve, on each arc, the scalar non-local equation:
\begin{equation}\label{hyper-stat-interval}
  -D_{i}\phi_{i,xx}=a_{i}\exp(\phi_{i}/\lambda_{i}^2)\left( C_{i}-\ds\frac{v_{i}}{\lambda_{i}^2}\int_{0}^x \exp(-\phi_{i}(y)/\lambda_{i}^2) dy \right)-b_{i}\phi_{i},
\end{equation}
which has  to be coupled at each node by the boundary conditions \eqref{Neumann}, \eqref{Neumann-phi}, \eqref{transmission}, and \eqref{transmission-phi}.

We can prove easily that in the case of dissipative coefficients  $\xi_{i,j}$ satisfying \eqref{condition_lambda},  \eqref{dissip-coeff} and the condition $\xi_{i,j}>0$, if  all the fluxes $v_{i}$ are null, then the density $u$ is continuous at a node, namely at a node $p$, the functions $u_{i}, \, i \in I_{p} {\cup} \,  O_{p}$  have all the same values. However, this is not the general case. 

\begin{figure}[htbp!]
\begin{center}
\includegraphics[scale=0.3]{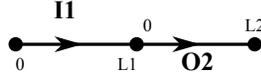}
\caption{One incoming and one outgoing arc connected at a node.}
\label{fig:rete}
\end{center}
\end{figure}
For the simplest network composed of one incoming $I=\{1\}$ and one outgoing $O=\{2\}$ arc, represented in Fig. \ref{fig:rete}, we find on each interval that
 $v_{1}=v_{2}=0$  from condition \eqref{Neumann}, and so we obtain the following local system for $\phi_{1}$ and $\phi_{2}$~:
\begin{equation}\label{hyper-stat}
 \left\{\begin{array}{l}
  -D_{1}\phi_{1,xx}=a_{1} C_{1}\exp(\phi_{1}/\lambda_{1}^2)-b_{1}\phi_{1},\\
  -D_{2}\phi_{2,xx}=a_{2} C_{2}\exp(\phi_{2}/\lambda_{2}^2)-b_{2}\phi_{2},
   \end{array}\right.
\end{equation}
with boundary conditions \eqref{Neumann-phi} and  \eqref{transmission-phi} for $\phi_{1}$ and $\phi_{2}$, which reads
\begin{equation*}
\partial_{x}\phi_{1}(L_{1})=\partial_{x}\phi_{2}(0)=\kappa_{1,2} (\phi_{2}(0)-\phi_{1}(L_{1})) , 
\end{equation*}
and
\begin{equation*}
\partial_{x}\phi_{1}(0)=\partial_{x}\phi_{2}(L_{2})=0.
\end{equation*}
We have also to take into account the following condition given by transmission condition \eqref{transmission}~:
$$\lambda_{2} \xi_{2,1} C_{1}\exp(\phi_{1}(L_{1})/\lambda_{1}^2) =\lambda_{1}\xi_{1,2} C_{2} \exp(\phi_{2}(0)/\lambda_{2}^2).$$

Solving the corresponding system  for $\phi_{1}$ and $\phi_{2}$ is a difficult task, even numerically, since an infinite number of solutions exist both for  $\phi_{1}$ and $\phi_{2}$, as in the case of a single interval \cite{GMNR09}, and it should be necessary to make them verify the above conditions at node.  In order to simplify our study, we limit  ourselves to state a result in the case of constant (in space)  stationary solutions  to system \eqref{hyper-gen}.

\begin{proposition}
Let us consider a general network $G=(\mathcal{N},\mathcal{A})$ and system  \eqref{hyper-gen} set on each arc of the network, complemented with boundary and transmission conditions \eqref{Neumann}, \eqref{Neumann-phi}, \eqref{transmission}, and \eqref{transmission-phi}.

(i) For general values of transmission coefficients $\xi_{i,j}$ satisfying \eqref{condition_lambda}, there is no non trivial constant stationary solution.

(ii) For  the special case of  transmission coefficients $\xi_{i,j}$ satisfying  the dissipation  relations \eqref{condition_lambda} and \eqref{dissip-coeff} and of  the ratios $a_{i}/b_{i}$ being  equal to the same constant on each arc,  there exists a one-parameter stationary solution, which is constant by arc.
\end{proposition}
\begin{proof}
Take a constant (in space) stationary solution to system \eqref{hyper-gen}. This means  that on each arc of the network, we have  three constant values $(u_{i}, v_{i}, \phi_{i})$,  which satisfy $\ds   v_{i}=0$, since $v_i=u_i \phi_{ix}=0$,  $\ds a_{i}u_{i}=b_{i}\phi_{i}$,
and boundary conditions \eqref{transmission},  \eqref{transmission-phi}, which become in that case
\begin{equation} \label{transmission-stat}
 u_{i}=\ds\sum_{j \in I_{p}  {\cup} O_{p}} \xi_{i,j} u_{j}, 
 \end{equation}
 and 
 \begin{equation}\label{transmission-phi-stat}
0=\ds \sum_{j \neq i}  \kappa_{i,j}(\phi_{j}-\phi_{i}).
\end{equation}
We remark that conditions \eqref{Neumann} and \eqref{Neumann-phi} are automatically satisfied.

(i) Denoting by $N$ the number of arcs of the network, we have to fix therefore $N$ unknowns to determine the stationary solution. Conditions \eqref{transmission-stat} -- \eqref{transmission-phi-stat} impose 4 equations by arc, unless the arc is connected to an outer node. In that case, there are only 2 conditions. To sum up, if we denote by $N_{out}$ the number of  outer nodes, we need to satisfy $4N-2N_{out}$ conditions. Taking into account relations \eqref{condition_lambda}, we obtain that equations \eqref{transmission} are linked and the system can be reduced to a system of  $4N-2N_{out}-N_{in}$ conditions, where $N_{in}$ is the number of inner nodes, which is,  generally speaking, greater than the number of unknowns. 
Therefore, unless some particular sets of coefficients $  \kappa_{i,j}$ and $\xi_{i,j}$, the only solution for previous system is the null one on each arc. 

(ii) Now, let us consider transmission coefficients $\xi_{i,j}$ satisfying relations \eqref{condition_lambda} and \eqref{dissip-coeff}.  
We also assume that there exists a constant $\alpha$ such that, for all $i$, we have $a_{i}=\alpha b_{i}$. 
In that case, we can find a stationary solution defined on each arc by $(U, 0, \alpha U)$. Such kind of solution satisfies clearly the transmission condition \eqref{transmission-phi-stat}, but satisfies also condition \eqref{transmission-stat} with relations \eqref{dissip-coeff}. 
\end{proof}
In the case (i) of the previous proposition, since the total initial mass is strictly positive and is preserved in time, we cannot  expect the system to converge  asymptotically to a stationary state which is constant on each arc and so non-constant asymptotic solutions are expected. In the case (ii), the constant state can be reached, and $U$ is determined by the total mass of the initial data.

\section{Numerical schemes}\label{sec3}

Here we introduce our numerical schemes.  We first  give some details about schemes for system \eqref{hyper-gen} on a single interval and the discretization of boundary conditions presented in \cite{NaRi}. Therefore, our main goal  will be to  generalize these schemes  to the case of a network. In the two first subsections, we will concentrate on the discretization of the hyperbolic part, whereas the discretization of the parabolic part will be treated in subsection \ref{parabolic}.


\subsection{Short review of the results from \cite{NaRi} about AHO schemes for system \eqref{hypergen2}  on a single interval }
 Let us consider a fixed single interval $[0,L]$.
We define a numerical grid using the following notations: $h$ is the space grid size, $k$ is the time grid size and $(x_{j}, t_n)=(j h, n k )$   for $j=0,\ldots, M+1, n \in\mathbb{N}$ are the grid points. 
In this subsection, we denote by $w^{n,j}$ the discretization of function $w$ on the grid at time $t_{n}$ and at point $x_{j}$ for $j=0,\ldots, M+1$ and $n\ge 0$. 
We
also use the notation $f^{n,j}$ for  $f(x_{j}, t_n)$, where $f$ is an explicitly known  function  depending on  $(x,t)$. 
Here we describe the discretization of system \eqref{hyper-gen}  
with no--flux boundary conditions $\ds v(0,t)=v(L,t)=0$, denoting by $f=\phi_x\,u$ and omitting the parabolic equation for $\phi$. 
 Since we also work with Neumann boundary conditions for the $\phi$ function, the function $f$ will satisfy the following conditions on the boundary~:
$\ds f(0,t)=f(L,t)=0$. We therefore consider the following system  
 \begin{equation}\label{hypergen2}
 \left\{\begin{array}{ll}
  u_t +v_x =0,\\ 
  v_t +\lambda^2 u_x 
   =f-v
 \end{array}\right.
\end{equation}
and rewrite it in a   diagonal  form, using the  usual change of variables  \eqref{cvar},
\begin{equation}
\label{systeme}
\left\{
\begin{aligned}
u^-_{t}-\lambda u^-_{x}&=\frac{1}{2}(u^+-u^-)-\frac{1}{2\lambda}f, \\
u^+_{t}+\lambda u^+_{x}&=\frac{1}{2}(u^--u^+)+\frac{1}{2\lambda}f.
\end{aligned}
\right.
\end{equation}
Set $\ds \omega=\left(\begin{array}{c} u^- \\ u^+ \end{array} \right)$, so that 
we can rewrite the system in vector form 
\begin{equation}
\label{systeme-vect}
\omega_{t}+\Lambda \omega_{x}=B\omega+F ,
\end{equation}
with $\ds \Lambda= \left(\begin{array}{cc}
-\lambda & 0 \\
0 & \lambda
\end{array}
\right) , \, B= \frac{1}{2}\left(\begin{array}{cc}
-1 & 1 \\
 1 & -1 
\end{array}
\right)$ and $\ds F=\frac{1}{2\lambda}\left(\begin{array}{c} -f \\ f \end{array} \right)$.
As shown in  \cite{NaRi}, to have a reliable scheme, with a correct resolution of fluxes at equilibrium, we have to deal with  Asymptotically High Order schemes   in the following  form~:
\begin{equation}
\label{scheme-aho}
\frac{\omega^{n+1,i}-\omega^{n,i}}{k}+
\frac{\Lambda}{2h}\left(\omega^{n,i+1}-\omega^{n,i-1}\right)-\frac{\lambda}{2h} (\omega^{n,i+1}-2\omega^{n,i}+\omega^{n,i-1})
=\sum_{\ell=-1,0,1} B^{\ell} \, \omega^{n,i+\ell}+ \sum_{\ell=-1,0,1} D^{\ell} \,  F^{n,i+\ell}.
\end{equation}
With the following choice of the matrices
\begin{equation}
\label{roe}
\begin{aligned}
& B^{0}= \frac{1}{4}\left(\begin{array}{cc}
-1 & 1 \\
 1 & -1 
\end{array}
\right) , \, B^{1}= \frac{1}{4}\left(\begin{array}{cc}
-1 & 1 \\
0 & 0
\end{array}
\right) , \, B^{-1}= \frac{1}{4}\left(\begin{array}{cc}
0 & 0 \\
 1 & -1
\end{array}
\right),\\
& D^{0}= \frac{1}{2}\left(\begin{array}{cc}
1 & 0 \\
 0 & 1 
\end{array}
\right) , \, D^{-1}= \frac{1}{2}\left(\begin{array}{cc}
0 & 0 \\
 0 & 1
\end{array}
\right) , \, D^{1}= \frac{1}{2}\left(\begin{array}{cc}
1& 0 \\
0 & 0
\end{array}
\right),
\end{aligned}
\end{equation}
we have a second--order  AHO scheme on every stationary solutions, which is enough to balance the flux of the system at equilibrium. This means that the scheme is second order when evaluated on stationary solutions. 
Monotonicity conditions are satisfied if $h\leq 4 \lambda$ and $\ds k \leq \frac{4h}{h+4\lambda}$, see  \cite{NaRi} for more details. Let us mention that it should be easy to consider third--order AHO schemes, but for simplicity (these schemes require a  fourth--order  AHO scheme for the parabolic equation with a five-points discretization for $\phi_{x}$), we prefer to limit our presentation to the second--order case. 

Boundary conditions for scheme  \eqref{scheme-aho} have to be treated carefully, to enforce mass-conservation.  In \cite{NaRi}, the following boundary conditions were used~: 
\begin{equation}\label{cond-bord1}
\begin{aligned}
v^{n+1,0}&=v^{n+1,M+1}
=0,
\\
u^{n+1,0}&
=\left(1-\lambda\frac{k}{h}\right)u^{n,0}
+\lambda\frac{k}{h}u^{n,1}
-k\left(\frac{1}{h}-\frac{1}{2\lambda}\right) v^{n,1}-\frac{k}{2\lambda} f^{n,1} ,
\\
u^{n+1,M+1}&
=\left(1-\lambda\frac{k}{h} \right)u^{n,M+1}
+\lambda\frac{k}{h}u^{n,M}
+k\left(\frac{1}{h}-\frac{1}{2\lambda}\right) v^{n,M} 
+ \frac{k}{2\lambda}
 f^{n,M}.
\end{aligned}
\end{equation}
These boundary conditions have been obtained by calculating the difference of the discrete mass at two successive computational  times and defining $u^{n+1,0}$ and $u^{n+1,M+1}$ as a function of the discrete quantities computed at time $t_{n}$ in order to cancel exactly this difference. Consequently, the discrete mass will be preserved in time  as the continuous mass $\ds  \int_{0}^{L} u(x,t) dx$ is conserved for system \eqref{hypergen2} with boundary conditions $\ds v(0,t)=v(L,t)=0$,  at the continuous level.
 This technique will be generalized   in this paper to the case of a network.

\subsection{The AHO scheme for system \eqref{hypergen2}  in the case of a network.}\label{AHONetwork}
 Let us consider a network as previously defined in Section \ref{sec2}. Each arc $a_{i}\in \mathcal{A}, \, 1\leq i \leq N$, is parametrized as an interval $a_{i}=[0,L_{i}]$ and is discretized with a space step $h_{i}$ and discretization  points $x^{j}_{i}$ for $j=0,\ldots, M_i+1$. We still denote by $k$ the time step, which is the same for all the arcs of the network.
In this subsection, we denote by $w^{n,j}_{i}$ the discretization on the grid at time $t_{n}$ and at point $x^{j}_{i}$  of a function $w_i, \ i=1,\ldots,N$ on the $i$-th arc for $j=0,\ldots, M_i+1$ and $n\ge 0$.

Now, we consider the AHO scheme \eqref{scheme-aho} on each interval,  and we rewrite it in the $u-v$ variables thanks to the change of variables \eqref{cvar},  in order to define the discrete boundary and transmission conditions.  We keep the possibility to use different AHO schemes on different intervals and therefore the coefficients of the scheme will be indexed by the number of the arc.
 Let  $\ds R=\left(\begin{array}{cc} 
 1 & 1 \\
 -\lambda & \lambda
\end{array}
\right) 
$ be the matrix associated to the change of variables \eqref{cvar}, namely such that $\ds \left(\begin{array}{c}
u  \\
 v \end{array}
\right)=R\left(\begin{array}{c}
  u^-\\
 u^+ \end{array}
\right)
$.
We rewrite  \eqref{scheme-aho} in the  variables $u$ and $v$  as~:
  \begin{equation}
\label{scheme-uv-interval}
\begin{aligned}
&u^{n+1,j}_{i}=u^{n,j}_{i}-\frac{k}{2h_{i}} \left(v^{n,j+1}_{i}-v^{n,j-1}_{i}\right)+ \frac{\lambda_{i}k}{2h_{i}}(u^{n,j+1}_{i}-2u^{n,j}_{i}+u^{n,j-1}_{i}) +\frac{k}{2}\Biggl(\sum_{\ell=-1,0,1} \beta^{\ell}_{u,u,i} u^{n,j+\ell}_{i}\\
&+\frac{1}{\lambda_{i}} \sum_{\ell=-1,0,1}  \beta^{\ell}_{u,v,i} v^{n,j+\ell}_{i}
 +\frac{1}{\lambda_{i}}\sum_{\ell=-1,0,1} \gamma^{\ell}_{u,i} f^{n,j+\ell}_{i}
 \Biggr),\\
&  v^{n+1,j}_{i}=v^{n,j}_{i} -\frac{\lambda^2_{i} k}{2h_{i}} \left(u^{n,j+1}_{i}-u^{n,j-1}_{i}\right)+ \frac{\lambda_{i}k}{2h_{i}}(v^{n,j+1}_{i}-2v^{n,j}_{i}+v^{n,j-1}_{i})
+\frac{k}{2}\Biggl(\lambda_{i}\sum_{\ell=-1,0,1} \beta^{\ell}_{v,u,i} u^{n,j+\ell}_{i} \\
&+\sum_{\ell=-1,0,1} \beta^{\ell}_{v,v,i} v^{n,j+\ell}_{i}+\sum_{\ell=-1,0,1} \gamma^{\ell}_{v,i} f^{n,j+\ell}_{i} \Biggr),
\end{aligned}
\end{equation}
with coefficients $ \beta^{\ell}_{u,u} , \,  \beta^{\ell}_{u,v}, \,  \beta^{\ell}_{v,u}, \,  \beta^{\ell}_{v,v} $ and $\gamma^{\ell}_{u},\,  \gamma^{\ell}_{v}$ defined by 
\begin{equation}
 \label{coeff}
RB^{\ell}R^{-1}=\frac 1  2 \left(\begin{array}{cc}
 \beta^{\ell}_{u,u} & \beta^{\ell}_{u,v}/\lambda \\
\lambda \beta^{\ell}_{v,u} & \beta^{\ell}_{v,v}
\end{array}
\right), \, RD^{\ell}R^{-1}=\frac 1  2 \left(\begin{array}{cc}
 * & \gamma^{\ell}_{u}/\lambda \\
* & \gamma^{\ell}_{v}
\end{array}
\right).
\end{equation}

Now, we   define the numerical boundary  conditions associated to this scheme.
As before for equation \eqref{cond-bord1}, we need four boundary or transmission conditions to implement this scheme on each interval. 
Considering an arc and its initial and end nodes, there are two possibilities: either they are external nodes, namely nodes from the outer boundaries linked to only one arc, or they are internal nodes connecting several arcs together. The boundary and transmission conditions will therefore depend on this feature.  Below, we will impose two boundary conditions  \eqref{bound1}--\eqref{bound3}   at outer nodes,  and two transmission conditions \eqref{bound2}--\eqref{bound4}  at inner nodes.

 The first type of boundary conditions will come from condition \eqref{Neumann} at outer nodes~:
\begin{equation}\label{bound1}
\left\{
\begin{aligned}
v_{i}^{n+1,0}=0, \textrm{ if } i\in I_{out}, \\
v_{i}^{n+1,M_{i}+1}=0, \textrm{ if } i\in O_{out},
\end{aligned}\right.
\end{equation} 
where $I_{out}$  (resp.  $O_{out}$) means that the arc is incoming from (resp. outgoing to) the outer boundary. 
The second one will come from a discretization of  the transmission condition \eqref{transmission} at  node $p$, that is to say
 \begin{equation}\label{bound2}
\left\{
\begin{aligned}
u_{-,i}^{n,M_{i}+1}=\sum_{j \in I_{p}} \xi_{i,j} u_{+,j}^{n,M_{j}+1}+\sum_{j \in O_{p}} \xi_{i,j} u_{-,j}^{n,0} ,&  \textrm{ if } i\in I_{p}, \\
u_{+,i}^{n,0}=\sum_{j \in I_{p}} \xi_{i,j} u_{+,j}^{n,M_{j}+1}+\sum_{j \in O_{p}} \xi_{i,j} u_{-,j}^{n,0} , & \textrm{ if } i\in O_{p}.
\end{aligned}
\right.
\end{equation}
However, these relations link all the unknowns together and they cannot be used alone. An effective way to compute all these quantities will be presented after equation \eqref{bound4} below.
We still have two missing conditions per arc, which can be recovered by imposing  the  exact mass conservation between two successive computational steps. The discrete total mass is given by $\ds \BIT=\sum_{i=1}^N \BI_{i}^n$,  where the mass corresponding to the arc $i$ is defined as:
\begin{equation}\label{mass-discr1}
 \BI_{i}^n=h_{i}\left(\frac{u_{i}^{n,0}}{2}+\sum_{j=1}^{M_{i}} u_{i}^{n,j}+\frac{u_{i}^{n,M_{i}+1}}{2}\right).
\end{equation}
Computing $\ds \BITp-\BIT$, we find:
\begin{equation*}
\begin{split}
&\BITp-\BIT=\sum_{i=1}^N \frac{h_{i}k}{2} \Biggl( \frac{1}{k}(u_{i}^{n+1,0}-u_{i}^{n,0})+\frac{1}{h_{i}} (v_{i}^{n,1}+v_{i}^{n,0})+ \frac{\lambda_{i}}{h_{i}}(u_{i}^{n,0} -u_{i}^{n,1})
+ \beta^{-1}_{u,u,i} u_{i}^{n,0} - \beta^{1}_{u,u,i} u_{i}^{n,1}\\
&- \frac{1}{\lambda_{i}} \beta^{1}_{u,v,i} v_{i}^{n,1}
+ \frac{1}{\lambda_{i}}\beta^{-1}_{u,v,i} v_{i}^{n,0} -\frac{1}{\lambda_{i}}\left(\gamma^{1}_{u,i}f_{i}^{n,1}- \gamma^{-1}_{u,i}f_{i}^{n,0}\right) 
\Biggr)\\
  &+\frac{h_{i}k}{2} \Biggl( \frac{1}{k}(u_{i}^{n+1,M_{i}+1}-u_i^{n,M_{i}+1})
  - \frac{1}{h_{i}} (v_i^{n,M_{i}}+v_i^{n,M_{i}+1})+ \frac{\lambda_{i}}{h_{i}}(u_i^{n,M_{i}+1}-u_i^{n,M_{i}})  +\beta^{1}_{u,u,i} u_i^{n,M_{i}+1}- \beta^{-1}_{u,u,i} u_i^{n,M_{i}}\\
  &
-\frac{1}{\lambda_{i}}\beta^{-1}_{u,v,i} v_i^{n,M_{i}}+\frac{1}{\lambda_{i}}\beta^{1}_{u,v,i} v_i^{n,M_{i}+1} + \frac{1}{\lambda_{i}}\left(\gamma^{1}_{u,i} f_{i}^{n,M_{i}+1} -\gamma^{-1}_{u,i} f_{i}^{n,M_{i}}\right)
\Biggr).
\end{split}
\end{equation*}
We  are going to impose boundary conditions such that the right-hand side in the previous difference is exactly canceled. 
On the outer boundaries we obtain the following type of boundary conditions, following equation \eqref{cond-bord1}~:
\begin{equation}
 \label{bound3}
 \left\{
\begin{aligned}
& u_{i}^{n+1,0}=\left(1-\lambda_{i}\frac{k}{h_{i}}-k\beta^{-1}_{u,u,i} \right)u_{i}^{n,0}
+k\left(\frac{\lambda_{i}}{h_{i}}+\beta^{1}_{u,u,i} \right)u_{i}^{n,1}
-k\left(\frac{1}{h_{i}}-\frac{\beta^{1}_{u,v,i}}{\lambda_{i}}\right) v_{i}^{n,1}\\
&
+\frac{k}{\lambda_{i}}\bigl(\gamma^{1}_{u,i}f_{i}^{n,1}- \gamma^{-1}_{u,i}f_{i}^{n,0}\bigr) 
,   \textrm{ if } i \in I_{out},\\
& u_{i}^{n+1,M_{i}+1}=\left(1-\lambda_{i}\frac{k}{h_{i}}-k\beta^{-1}_{u,u,i} \right)u_{i}^{n,M_{i}+1}
+k\left(\frac{\lambda_{i}}{h_{i}}+\beta^{1}_{u,u,i} \right)u_{i}^{n,M_{i}} 
+k\left(\frac{1}{h_{i}}-\frac{\beta^{-1}_{u,v,i}}{\lambda_{i}}\right) v_{i}^{n,M_{i}} \\
&
- \frac{k}{\lambda_{i}}\left(\gamma^{1}_{u,i} f_{i}^{n,M_{i}+1} -\gamma^{-1}_{u,i} f_{i}^{n,M_{i}}\right), \textrm{ if } i \in O_{out},
\end{aligned}
\right.
\end{equation}
where $I_{out}$ and $O_{out}$ have the same meaning as previously. These expressions correspond to boundary conditions \eqref{cond-bord1} in the case of a more general AHO scheme \cite{NaRi}.
Then, using the conditions (\ref{bound3}) to simplify the computation of  $\ds \BITp-\BIT$ and summing with respect to the nodes instead of the arcs, we can rewrite the remaining difference of mass in $\ds u^{\pm}$ variables as:
\begin{equation*}
\begin{split}
&\BITp-\BIT=\sum_{p\in \mathcal{N}}\sum_{i \in O_{p}} \frac{h_{i}k}{2} \Biggl( \frac{1}{k}u_{+,i}^{n+1,0}+\frac{1}{k}u_{-,i}^{n+1,0}
+u_{+,i}^{n,0}\bigl( - \frac{1}{k}+ 2\frac{\lambda_{i}}{h_{i}}+\beta^{-1}_{u,u,i}
+\beta^{-1}_{u,v,i}\bigr)
 +u_{-,i}^{n,0}\bigl( - \frac{1}{k}+\beta^{-1}_{u,u,i}\\
&
-\beta^{-1}_{u,v,i}\bigr)-u_{+,i}^{n,1}\bigl( \beta^{1}_{u,u,i}+\beta^{1}_{u,v,i}\bigr)
+u_{-,i}^{n,1}\bigl(-2\frac{\lambda_{i}}{h_{i}}-\beta^{1}_{u,u,i}+\beta^{1}_{u,v,i}\bigr)-\frac{1}{\lambda_{i}}(\gamma^{1}_{u,i}f_{i}^{n,1}
- \gamma^{-1}_{u,i}f_{i}^{n,0}) \Biggr)\\
& +\sum_{p\in \mathcal{N}}\sum_{i\in I_{p}}\frac{h_{i}k}{2} \Biggl( 
 \frac{1}{k}u_{+,i}^{n+1,M_{i}+1}+\frac{1}{k}u_{-,i}^{n+1,M_{i}+1}
+u_{+,i}^{n,M_{i}+1}\bigl( - \frac{1}{k}
+\beta^{1}_{u,u,i}+\beta^{1}_{u,v,i}\bigr)
 +u_{-,i}^{n,M_{i}+1}\bigl( - \frac{1}{k}+2\frac{\lambda_{i}}{h_{i}}+\beta^{1}_{u,u,i}\\
&-\beta^{1}_{u,v,i}\bigr)
-u_{+,i}^{n,M_{i}}(2\frac{\lambda_{i}}{h_{i}}+\beta^{-1}_{u,u,i}+\beta^{-1}_{u,v,i})+u_{-,i}^{n,M_{i}}\left(-\beta^{-1}_{u,u,i}+\beta^{-1}_{u,v,i}\right) + \frac{1}{\lambda_{i}}(\gamma^{1}_{u,i} f_{i}^{n,M_{i}+1} -\gamma^{-1}_{u,i} f_{i}^{n,M_{i}})\Biggr).
\end{split}
\end{equation*}
Therefore, using the transmission conditions \eqref{bound2} for $\ds u_{-,i}^{n+1,M_{i}+1}$ if $i\in I_{p}$ and for $\ds u_{+,i}^{n+1,0}$ if $i\in O_{p}$, we can split the equation interval by interval and obtain the following numerical boundary conditions:
\begin{equation}
 \label{bound4}
\begin{aligned}
& u_{+,i}^{n+1,M_{i}+1}=h_{i}\left(h_{i}+  \sum_{j \in I_{p}  {\cup} O_{p}} h_{j}\xi_{j,i}\right)^{-1} \times \Bigl(u_{+,i}^{n,M_{i}+1}( 1 -k\beta^{1}_{u,u,i}-k\beta^{1}_{u,v,i})
+u_{-,i}^{n,M_{i}+1}\bigl( 1-2k\frac{\lambda_{i}}{h_{i}}-k\beta^{1}_{u,u,i} \\
&+k\beta^{1}_{u,v,i}\bigr)
+k u_{+,i}^{n,M_{i}}\bigl(2\frac{\lambda_{i}}{h_{i}}+ \beta^{-1}_{u,u,i}+\beta^{-1}_{u,v,i}\bigr)
+k u_{-,i}^{n,M_{i}}(\beta^{-1}_{u,u,i} -\beta^{-1}_{u,v,i})- \frac{k}{\lambda_{i}}\left(\gamma^{1}_{u,i} f_{i}^{n,M_{i}+1} -\gamma^{-1}_{u,i} f_{i}^{n,M_{i}}\right)\Bigr), \textrm{ if } i\in I_{p},
\end{aligned}
\end{equation}
\begin{equation*}
\begin{aligned}
& u_{-,i}^{n+1,0}=h_{i}\left(h_{i}+ \sum_{j  \in I_{p}   {\cup} O_{p}}h_{j}\xi_{j,i} \right)^{-1} \times \Bigl(u_{+,i}^{n,0}( 1-2k \frac{\lambda_{i}}{h_{i}}-k\beta^{-1}_{u,u,i}
-k\beta^{-1}_{u,v,i})
+u_{-,i}^{n,0}\bigl( 1-k\beta^{-1}_{u,u,i}+k\beta^{-1}_{u,v,i}\bigr)\\
&+k u_{+,i}^{n,1}(\beta^{1}_{u,u,i}
+ \beta^{1}_{u,v,i})
+k u_{-,i}^{n,1}\bigl(2\frac{\lambda_{i}}{h_{i}}+\beta^{1}_{u,u,i}-\beta^{1}_{u,v,i}\bigr)
+\frac{k}{\lambda_{i}}(\gamma^{1}_{u,i}f_{i}^{n,1}- \gamma^{-1}_{u,i}f_{i}^{n,0}) 
\Bigr)
, \textrm{ if } i\in O_{p}.
\end{aligned}
\end{equation*}
Once these quantities are computed, we can use equations \eqref{bound2} at time $t_{n+1}$, to obtain $u_{-,i}^{n+1,M_{i}+1}$ if $i \in I_{p}$  and $u_{+,i}^{n+1,0}$ if $i \in O_{p}$.

In conclusion, we have imposed four boundary conditions \eqref{bound1},  \eqref{bound2}, \eqref{bound3}, and \eqref{bound4} on each interval. Conditions \eqref{bound1} and \eqref{bound3} deal with the outer boundary and are written in the $u-v$ variables, whereas conditions \eqref{bound2} and \eqref{bound4} deal with the node and are written in the $u^{\pm}$ variables. Under these conditions, the total numerical mass is conserved at each step. 

Now, we have to discuss the consistency of all these conditions. First, conditions \eqref{bound1},  \eqref{bound2} are imposed exactly. Besides, it has been proved in \cite{NaRi} that conditions  \eqref{bound3}, set on the outer boundary, are generally of order one and  of order two on  stationary solutions. Finally, we need to consider the consistency of the conditions  \eqref{bound4} at node.  We present here only the case $i \in O_{p}$. Expanding in Taylor series up to order one, we get:
\begin{equation*}
\begin{aligned}
& u_{-,i}^{n+1,0}-\bigl(1+ \sum_{j  \in I_{p}  {\cup} O_{p}}\frac{h_{j}}{h_{i}}\xi_{j,i} \bigr)^{-1} \times \Bigl(u_{+,i}^{n,0}( 1-2k \frac{\lambda_{i}}{h_{i}}-k\beta^{-1}_{u,u,i}
-k\beta^{-1}_{u,v,i})+u_{-,i}^{n,0}\bigl( 1-k\beta^{-1}_{u,u,i}+k\beta^{-1}_{u,v,i}\bigr)\\
&
+k u_{+,i}^{n,1}(\beta^{1}_{u,u,i}+ \beta^{1}_{u,v,i})+k u_{-,i}^{n,1}\bigl(2\frac{\lambda_{i}}{h_{i}}+\beta^{1}_{u,u,i}-\beta^{1}_{u,v,i}\bigr)
+\frac{k}{\lambda_{i}}(\gamma^{1}_{u,i}f_{i}^{n,1}- \gamma^{-1}_{u,i}f_{i}^{n,0}) \Bigr) \\
&= u_{-,i}^{n,0} \Bigl(1-\bigl(1+ \sum_{j  \in I_{p}  {\cup} O_{p}}\frac{h_{j}}{h_{i}}\xi_{j,i} \bigr)^{-1} \bigl( 1+2k\frac{\lambda_{i}}{h_{i}}\bigr)
 \Bigr) -u_{+,i}^{n,0}\bigl(1+ \sum_{j  \in I_{p} {\cup} O_{p}}\frac{h_{j}}{h_{i}}\xi_{j,i} \bigr)^{-1}\Bigl( 1-2k \frac{\lambda_{i}}{h_{i}}\Bigr)\\
 & +O(k+\sum_{i \in I_{p}  {\cup} O_{p}} h_{i}).\\
\end{aligned}
\end{equation*}
Now, to have consistency, namely to cancel the last two terms on the R.H.S., we  need to impose the following condition linking the space   and the time step on each arc :
\begin{equation}\label{cfl}
h_{i}=2 k \lambda_{i}, 
\end{equation}
 which implies, thanks to  \eqref{condition_lambda}: 
\[   \left(1+ \sum_{j  \in I_{p} {\cup} O_{p}}\frac{h_{j}}{h_{i}}\xi_{j,i} \right)^{-1}=\frac 1 2.\]
Under this condition and using equations \eqref{systeme}, expanding in Taylor series up to order three  we find: 
\begin{equation*}
\begin{aligned}
& u_{-,i}^{n+1,0}-\displaystyle \bigl(1+ \sum_{j  \in I_{p} {\cup} O_{p}}\frac{h_{j}}{h_{i}}\xi_{j,i} \bigr)^{-1} \times \Bigl(u_{+,i}^{n,0}( 1-2k \frac{\lambda_{i}}{h_{i}}-k\beta^{-1}_{u,u,i}
-k\beta^{-1}_{u,v,i})+u_{-,i}^{n,0}\bigl( 1-k\beta^{-1}_{u,u,i}+k\beta^{-1}_{u,v,i}\bigr)\\
&
+k u_{+,i}^{n,1}(\beta^{1}_{u,u,i}+ \beta^{1}_{u,v,i})+k u_{-,i}^{n,1}\bigl(2\frac{\lambda_{i}}{h_{i}}+\beta^{1}_{u,u,i}-\beta^{1}_{u,v,i}\bigr)
+\frac{k}{\lambda_{i}}(\gamma^{1}_{u,i}f_{i}^{n,1}- \gamma^{-1}_{u,i}f_{i}^{n,0}) \Bigr) \\
&= \frac k 2 u_{-,i}^{n,0} (\beta^{1}_{u,v,i}+\beta^{-1}_{u,u,i}-\beta^{1}_{u,u,i}-\beta^{-1}_{u,v,i})+k  \partial_{t} u_{-,i}^{n,0}+\frac{k^2}{2} \partial_{tt} u_{-,i}^{n,0}-k\lambda_{i}\bigl(1 + k(\beta^{1}_{u,u,i}-\beta^{1}_{u,v,i}) \bigr)  \partial_{x} u_{-,i}^{n,0}\\
 &-k^2\lambda_{i}^2  \partial_{xx} u_{-,i}^{n,0} 
 + \frac k 2 u_{+,i}^{n,0}(\beta^{-1}_{u,u,i}+\beta^{-1}_{u,v,i}-\beta^{1}_{u,u,i}- \beta^{1}_{u,v,i})
 - k^2\lambda_{i}\bigl( \beta^{1}_{u,u,i}+\beta^{1}_{u,v,i} \bigr)  \partial_{x} u_{+,i}^{n,0}\\
&-\frac{k}{2\lambda_{i}}(\gamma^{1}_{u,i}-\gamma^{-1}_{u,i})f_{i}^{n,0}-k^2\gamma^{1}_{u,i}\partial_{x} f_{i}^{n,0}+O(k^3)\\
&= \frac k 2 \Bigl( u_{-,i}^{n,0} (-1+\beta^{1}_{u,v,i}+\beta^{-1}_{u,u,i}-\beta^{1}_{u,u,i}-\beta^{-1}_{u,v,i})+ u_{+,i}^{n,0}(1+\beta^{-1}_{u,u,i}+\beta^{-1}_{u,v,i}-\beta^{1}_{u,u,i}- \beta^{1}_{u,v,i})\\
&-\frac{1}{\lambda_{i}}(1+\gamma^{1}_{u,i}-\gamma^{-1}_{u,i})f_{i}^{n,0} \Bigr)
+k^2  \Bigl( \frac 1 2 \partial_{tt} u_{-,i}^{n,0} -\lambda_{i}  \partial_{tx} u_{-,i}^{n,0}
-\lambda_{i}(\frac 1 2 +\beta^{1}_{u,u,i}-\beta^{1}_{u,v,i})  \partial_{x} u_{-,i}^{n,0} \\
&
- \lambda_{i}( -\frac 1 2 +\beta^{1}_{u,u,i}+\beta^{1}_{u,v,i} )  \partial_{x} u_{+,i}^{n,0}
-(\gamma^{1}_{u,i}+\frac 1 2 )\partial_{x} f_{i}^{n,0}
 \Bigr)
+O(k^3).
\end{aligned}
\end{equation*}
Thanks to this development we can state our general result of consistency. 
\begin{proposition}\label{consistency}
Given a general scheme in the form \eqref{scheme-aho}, the conditions  \eqref{bound4} at node are consistent only if on each arc the condition \eqref{cfl} is verified. To have the 
second order accuracy at node  the following conditions on the coefficients of the scheme have to be verified: 
\begin{equation}\label{order2bound}
\beta^{1}_{u,u,i}=\beta^{-1}_{u,u,i}, \,   \beta^{1}_{u,v,i}-\beta^{-1}_{u,v,i}=1  , \, \gamma^{-1}_{u,i}-\gamma^{1}_{u,i}=1.
\end{equation}
Moreover, to have a third order accuracy for stationary solutions, we need~:
 \begin{equation}\label{order3boundstat}
\beta^{1}_{u,u,i}=\beta^{-1}_{u,u,i}=0, \, \beta^{1}_{u,v,i}=-\beta^{-1}_{u,v,i}=\frac 1 2, \, \gamma^{1}_{u,i}=-\gamma^{-1}_{u,i}=-\frac 1 2.
\end{equation}

\end{proposition}
Notice that, all these conditions are satisfied for the Roe scheme defined by   \eqref{roe}.

\subsection{Discretization of the parabolic equation for $\phi$ in system \eqref{hyper-gen-diag}}\label{parabolic}

Now, let us explain how to compute the approximations $f^{n+1,j}_{i}$
 of the function $f$ on the arc $i$ at discretization point $x^{j}_{i}$ and time $t_{n+1}$ needed for computing  \eqref{scheme-uv-interval},  \eqref{bound3}, and  \eqref{bound4}. 
 Referring to system \eqref{hyper-gen-diag}, we have  $\ds f=\phi_{x} u$, where $\phi$ satisfies the parabolic equation   $\ds \phi_{t}-D\, \phi_{xx}=au-b\phi$ on each arc.
 Boundary conditions for $\phi$ are given by equations \eqref{Neumann-phi}  on the outer boundary and  \eqref{transmission-phi} at a node.

We  solve the parabolic equation, using a finite differences scheme in space and a Crank-Nicolson method in time, namely an explicit-implicit method in time. 

Therefore, we will have the following equation for $\phi^{n,j}_{i}, \, 1 \leq j \leq M_{i}$,
\begin{equation}\label{discr-CN}
\begin{aligned}
\phi^{n+1,j}_{i}&=\phi^{n,j}_{i}-\frac{D_{i}k}{2h_{i}^2} \left( -\phi^{n,j+1}_{i}+2\phi^{n,j}_{i}-\phi^{n,j-1}_{i}\right)-\frac{D_{i}k}{2h_{i}^2} \Bigl( -\phi^{n+1,j+1}_{i}
+2\phi^{n+1,j}_{i}
-\phi^{n+1,j-1}_{i}\Bigr)\\
&+\frac{a_{i}k}{2}(u^{n+1,j}_{i}+u^{n,j}_{i})-\frac{b_{i}k}{2}(\phi^{n+1,j}_{i}+\phi^{n,j}_{i}).
\end{aligned}
\end{equation}
Now, let us find the two boundary conditions needed on each interval. As in subsection \ref{AHONetwork}, the boundary conditions will be given in the case of an outer node  and in the case of an inner node.
On the outer boundary, condition  \eqref{Neumann-phi} for $\phi$ is discretized using a second order approximation, which is
 \begin{equation}\label{discr-outer}
\left\{\begin{aligned}
&\phi_{i}^{n,0}= \frac 4 3 \phi^{n,1}_{i}-\frac 1 3\phi^{n,2}_{i}, \, \textrm{ if } i\in I_{out}, \\
&\phi_{i}^{n,M_{i}+1}= \frac 4 3 \phi^{n,M_{i}}_{i}-\frac 1 3\phi^{n,M_{i}-1}_{i}, \, \textrm{ if } i\in O_{out}.
\end{aligned}\right.
\end{equation}
Let us now describe our numerical approximation for the transmission condition  \eqref{transmission-phi} which,  as the transmission condition  for the hyperbolic part  \eqref{transmission}, couples the $\phi$ functions of arcs having  a node in common.
 
Condition \eqref{transmission-phi} is discretized using the same second-order discretization formula as before, namely  we have
at  node $p$,
\begin{align*}
& \phi_{i}^{n,M_{i}+1}=\frac{4}{3}\phi_{i}^{n,M_{i}}-\frac{1}{3}\phi_{i}^{n,M_{i}-1}+\frac{2}{3}\frac{h_{i}}{D_{i}}\sum_{j \in I_{p}} \kappa_{i,j} (\phi_{j}^{n,M_{j}+1}-\phi_{i}^{n,M_{i}+1})
+\frac{2}{3}\frac{h_{i}}{D_{i}}\sum_{j \in O_{p}} \kappa_{i,j} (\phi_{j}^{n,0}-\phi_{i}^{n,M_{i}+1}) , 
 \textrm{ if }i\in I_{p}, \\
&\phi_{i}^{n,0}=\frac{4}{3}\phi_{i}^{n,1}-\frac{1}{3}\phi_{i}^{n,2}+\frac{2}{3}\frac{h_{i}}{D_{i}}\sum_{j \in I_{p}} \kappa_{i,j} (\phi_{j}^{n,M_{j}+1}-\phi_{i}^{n,0})
+\frac{2}{3}\frac{h_{i}}{D_{i}}\sum_{j \in O_{p}} \kappa_{i,j} (\phi_{j}^{n,0}-\phi_{i}^{n,0}) ,  
 \textrm{ if }i\in O_{p}.
\end{align*}
These relations can be rewritten as~:
\begin{equation}\label{discr-transmi}
\begin{aligned}
& \underbrace{\left(1+\frac{2}{3}\frac{h_{i}}{D_{i}} \sum_{j\in I_{p} {\cup} O_{p}} \kappa_{i,j} \right)}_{=\eta_{i}^p}\phi_{i}^{n,M_{i}+1}=\frac{4}{3}\phi_{i}^{n,M_{i}}-\frac{1}{3}\phi_{i}^{n,M_{i}-1} 
+\frac{2}{3}\frac{h_{i}}{D_{i}}\sum_{j \in I_{p}} \kappa_{i,j} \phi_{j}^{n,M_{j}+1}
+\frac{2}{3}\frac{h_{i}}{D_{i}}\sum_{j \in O_{p}} \kappa_{i,j} \phi_{j}^{n,0} ,  
  \textrm{ if }i\in I_{p}, 
\\
& \underbrace{\left(1+\frac{2}{3}\frac{h_{i}}{D_{i}} \sum_{j\in I_{p}  {\cup} O_{p}} \kappa_{i,j} \right)}_{=\eta_{i}^p}\phi_{i}^{n,0}=\frac{4}{3}\phi_{i}^{n,1}-\frac{1}{3}\phi_{i}^{n,2}
+\frac{2}{3}\frac{h_{i}}{D_{i}}\sum_{j \in I_{p}} \kappa_{i,j} \phi_{j}^{n,M_{j}+1}+\frac{2}{3}\frac{h_{i}}{D_{i}}\sum_{j \in O_{p}} \kappa_{i,j} \phi_{j}^{n,0}, 
\textrm{ if }i\in O_{p}.
\end{aligned}
\end{equation}
Let us remark that the previous discretizations are compatible with relations \eqref{discr-outer} considering that  for outer boundaries the coefficients $\kappa_{i,j}$ are null. Therefore, in this case, the value of $\eta_{i}^{out}$  is just equal to  $1$.
Since equations  \eqref{discr-transmi} are coupling the unknowns of all arcs altogether, we have to solve  a large system which contains  all the equations of type \eqref{discr-CN} and also  the discretizations of transmission conditions \eqref{discr-transmi}.

Once the values of $\ds \phi^{n+1,j}_{i}$ are known, we can compute a  second-order discretization of the derivatives of $\phi$  which gives the values of the $f$ function, namely :
\begin{equation*}
\phi_{x,i}^{n+1,j}=  \left\{\begin{aligned}
& \ds\frac{1}{2\,h_{i}} \left(\phi ^{n+1,j+1}_{i}-\phi ^{n+1,j-1}_{i}\right),  1 \leq j \leq M_{i},\\
& \ds\frac{1}{2\,h_{i}} \left( -\phi ^{n+1,2}_{i}+4\phi ^{n+1,1}_{i}-3\phi ^{n+1,0}_{i}\right), j=0,\\
& \ds\frac{1}{2\,h_{i}} \left(\phi ^{n+1,M_{i}-1}_{i}-4\phi ^{n+1,M_{i}}_{i}+3\phi ^{n+1,M_{i}+1}_{i}\right), j=M_{i}+1.
\end{aligned}\right.
\end{equation*}
The discretization of $f$ needed at equations \eqref{scheme-uv-interval},\eqref{bound3}, and	\eqref{bound4} is therefore given by $f^{n+1,j}_{i}=\phi_{x,i}^{n+1,j}  u^{n+1,j}_{i}$.

\section{Numerical tests}\label{sec4}

Here we present some numerical experiments  for  system (\ref{hyper-gen}) on networks, with the use of the methods introduced in Section \ref{sec3}, namely the second--order AHO scheme  for the hyperbolic part, 
complemented with the Crank-Nicolson scheme   for the parabolic part.
We start with a simple test for the  AHO scheme on the hyperbolic part of Section \ref{sec3} in the case of a simplified system, where $\phi_{x} $ is  equal to a constant $\alpha$ on each arc, for which  we know the exact  stationary states.

\subsection{Case   $\phi_{x}$ constant.}  For this example, we omit the equation for $\phi$ so that the system becomes
\begin{equation}
\label{hyper-simple-diag-2}
\left\{
\begin{aligned}
u^+_{t}+\lambda u^+_{x}&=\frac{1}{2\lambda}\left((\alpha-\lambda)u^++(\alpha+\lambda)u^-\right),\\
u^-_{t}-\lambda u^-_{x}&=-\frac{1}{2\lambda}\left((\alpha-\lambda)u^++(\alpha+\lambda)u^- \right).
\end{aligned}
\right.
\end{equation}
This system  is suitable to test the accuracy of the numerical approximation, since it is easy to compute its asymptotic stationary solutions.
We also rewrite the previous system \eqref{hyper-simple-diag-2} using the usual change of variables \eqref{cvar} which gives 
\begin{equation}\label{hyper-simple}
 \left\{\begin{array}{ll}
  u_t +v_x =0,\\ 
  v_t +\lambda^2 u_x 
   = \alpha\,u-v,
 \end{array}\right.
\end{equation}
with $\alpha$ a constant. To satisfy the subcharacteristic condition in \cite{Na96}, we also assume that
\begin{equation}\label{condition}
\lambda> |\alpha|.
\end{equation}
Let us explain how to find the stationary states in the case of the two-arcs network of Figure \ref{fig:rete}. The method can be easily generalized to more complex networks.
In that case, the stationary solutions satisfy the following equations on the intervals $I_{1}$ and $I_{2}$~:
\begin{equation*}
 \left\{\begin{array}{l}
  v_{i,x} =0,\\ 
  \lambda_{i}^2 u_{i,x}
   =\alpha_{i}\,u_{i}-v_{i},
 \end{array}\right.
\end{equation*}
that is to say
\begin{equation}\label{hyper-simple-stat}
 \left\{\begin{array}{l}
  v_{i}=constant,\\ 
  u_{i}=C_{i} \exp(\alpha_{i} x/\lambda_{i}^2)+v_{i}/\alpha_{i}.
 \end{array}\right.
\end{equation}
Since 
both  intervals are connected to the outer boundary, due to boundary condition \eqref{Neumann}, we have  $\ds v_{1}=v_{2}=0$. Therefore we obtain non constant solutions on each arc, given  by 
$\ds  u^{\pm}_{i}= \frac{u_{i}}{2}=\frac{C_{i}}{2} \exp(\alpha_{i} x/\lambda_{i}^2) $ and the constants $C_{i}$ are computed thanks to condition \eqref{transmission}. 
Remark that, in that case, we do not expect to have asymptotic states given by constant stationary solutions, since the only possible constant solution is the null one, which will  be unsuitable, due to the constraint of the conservation of mass.  
Set
\begin{equation}\label{cctil}
\widetilde C_{1}=\frac{C_{1}}{\lambda_{1}} \exp(\alpha_{1} L_{1}/\lambda_{1}^2), \ \widetilde C_{2}=\frac{C_{2}}{\lambda_{2}}.
\end{equation}
These constants solve the following system~:
\begin{equation}\label{sist}
\mathbf{M}\mathbf{\widetilde C}=
\left(\begin{array}{cc}
\lambda_{1}(\xi_{1,1}-1) & \lambda_{2} \xi_{1,2} \\
\lambda_{1}\xi_{2,1} & \lambda_{2}(\xi_{2,2} -1) 
\end{array}\right)
 \left(\begin{array}{c}
 \widetilde C_{1} \\
 \widetilde C_{2}  
 \end{array}\right)=0.
\end{equation}
According to  \eqref{condition_lambda},  $\ds \textrm{Ker } \mathbf{M}\neq \{ 0\}$, and so we have at most one equation and two unknowns.
Therefore, there exists at least one family of non  trivial stationary solutions to system \eqref{hyper-simple}  and exactly one family  when  $\ds \textrm{dim } \textrm{Ker } \mathbf{M}=1$. Remark that in the general  case of a single node with an arbitrary number of incoming and outcoming arcs, assuming that all coefficients $\xi_{i,j}$ are strictly positive -- or more generally, that the matrix formed by these coefficients is irreducible, which is somewhat meaningful in the biological context --, we can prove  that we have exactly $\ds \textrm{dim } \textrm{Ker } \mathbf{M}=1$, thanks to the classical Perron-Frobenius theorem.

In the case we are looking for an asymptotic state as a stationary state of the system, we can also take into account the conservation for mass. In that case, the  stationary state we compute should have the same mass as the initial datum. More precisely, according to equation 
$$\mu_0 = \sum_{i=1}^{2} \int_{0}^{L_i} C_i \exp{\left(\frac{\alpha x}{{\lambda_i}^2}\right)} 
dx  = \sum_{i=1}^{2} C_i \frac{{\lambda^2_i}}{\alpha} \left(\exp{\left(\frac{\alpha L_i}{{\lambda_i}^2}\right)} - 1\right),$$ 
we have that 
the free parameter is fixed by the mass conservation.

In particular we set $L_1=4$, $L_2=1$, $\alpha_i=\alpha=0.5$, $\lambda_1=2$, $\lambda_2=1$ and take the dissipative transmission coefficients $\xi_{1,1}=0.8, \, \xi_{2,1}=0.4,\,  \xi_{1,2}= 0.2, \, \xi_{2,2}=0.6$. If $\mu_0=250$, the system is solved by $\widetilde C_{1}\sim 28.13$ and $\widetilde C_{2}\sim 56.25$, so that the stationary solutions are $u_{1}=  C_1\exp(x/8)$ and $u_{2}= C_2 \exp(x/2)$, with $C_1\sim 34.12$ and $C_2 \sim 56.25$. 
The numerical simulations provide the asymptotic densities plotted in Fig. \ref{alpha05} and we notice a nice agreement with the stationary solutions computed analytically. Remark that densities are continuous at the node as explained in Section  \ref{SS-staz} for dissipative coefficients and vanishing fluxes.

\begin{figure}[htbp!]
\includegraphics[width=8cm,height=7cm]{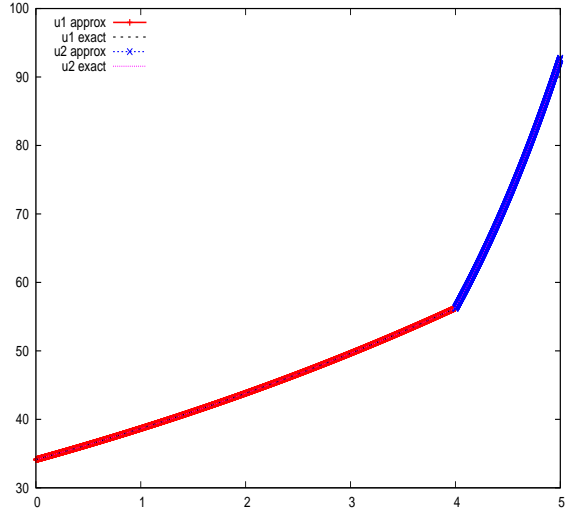}
\caption{Comparison between the densities of the exact and the numerical stationary solutions on arcs 1 and  2 obtained for $\lambda_1=2$, $\lambda_2=1$, $\alpha_i=\alpha=0.5$, initial mass $\mu_0=250$ distributed on the network as a symmetric perturbation of the value $50$, $L_1=4$, $ L_2=1$, dissipative coefficients $\xi_{1,1}=0.8, \, \xi_{2,1}=0.4, \, \xi_{1,2}= 0.2, \, \xi_{2,2}=0.6$ and time  $T=28$.}\label{alpha05}
\end{figure}

In Fig. \ref{fig:tab1} we present the log-log plot of the error in the $L^1$-norm computed using the  formula (\ref{error}) of Section 4, between the approximated and the asymptotic solutions to  system (\ref{hyper-simple}). The results in Fig. \ref{fig:tab1} show that the AHO approximation scheme provides the stationary solutions of the simplified hyperbolic model (\ref{hyper-simple}) with an accuracy of first order, and the error for the flux function $v$ tends clearly to zero, faster than for the function $u$.
\begin{figure}[htbp!]
\includegraphics[scale=0.5]{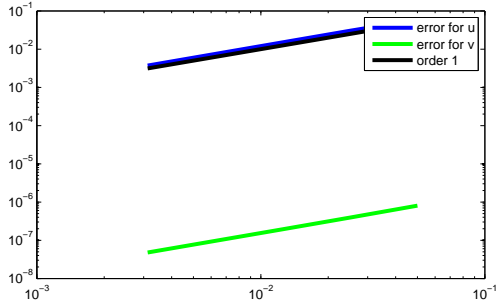} 
\caption{Log-log plot of the error in $L^1$ norm between the approximated and the the asymptotic solutions, as a function of the space step, for the solutions to system (\ref{hyper-simple}). The error is displayed in blue for  $u$, and in green for $v$. Initial data are distributed on the network as a symmetric perturbation of the value $50$. We used different space steps satisfying condition \eqref{cfl}, with $\lambda_1=2$, $\lambda_2=1$, $L_1=4, L_2=1$, $\mu_0= 250$, $T=100$.}
\label{fig:tab1}
\end{figure}

More examples and results showing the asymptotic behavior of solutions to the simple problem (\ref{hyper-simple}) on larger networks can be found in \cite{BNR11}, while some analytical results are given in \cite{GPhD}.

\subsection{Asymptotic solutions to the full system (\ref{hyper-gen-diag})}

Next, we deal with the full system (\ref{hyper-gen-diag}), which now include the chemotaxis equation. First, we consider again a network with only two arcs. We take the following data: the total mass $\mu_0=160$ distributed as a small perturbation of the value $20$ on two arcs of length $L_1=6$ and $L_2=2$, see Fig. \ref{indat}, $a_i=b_i=1$, $u_i(x,0)=\phi_i(x,0)$ and $v_i(x,0)=0$, $i=1,2$ and $\lambda_1=5, \lambda_2=4$. 
\begin{figure}
\includegraphics[width=5.5cm,height=5.5cm]{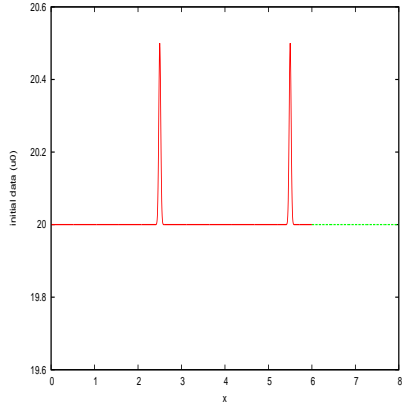}
\caption{Initial data corresponding to the total mass $\mu_0=160$. }\label{indat}
\end{figure}
In the next figures we represent the asymptotic stable solutions to system (\ref{hyper-gen-diag}) on the two-arcs network, produced by our  scheme. All the solutions are plotted at a time where the stationary state is already reached. In particular, in Fig. \ref{solcost} we plot a constant solution obtained using  the dissipative transmission coefficients of Section \ref{SS-dissip-trans}. In that case we can observe what was explained in Section \ref{SS-staz}, namely that in the case of two arcs and one node, there exist particular  dissipative transmission coefficients, such that the asymptotic stationary solutions are  constants on all the arcs. In Fig. \ref{solnoncost} we plot the more common case of   non-constant solutions,  obtained using different parameters and non dissipative coefficients. In both cases the limit flux function $v$ is equal to zero everywhere, since for the stationary solution the flux is constant, the flux on the external nodes is zero, and all the arcs are connected to external nodes.
\begin{figure}[htbp!]
\includegraphics[width=12cm,height=6.5cm]{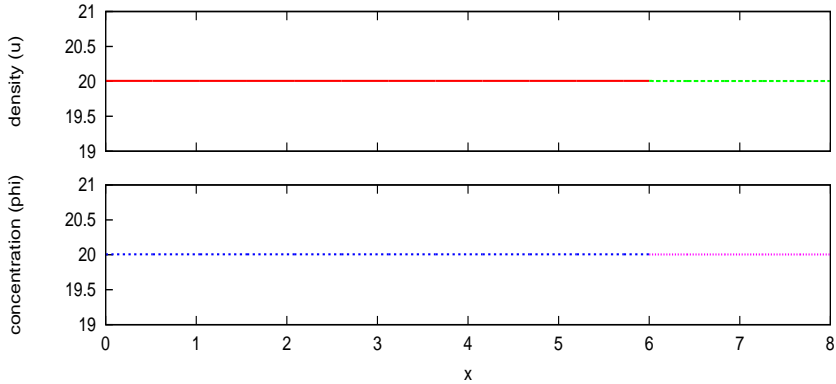}
\caption{Asymptotic solution for $\lambda_1=5, \lambda_2=4$, dissipative coefficients $\xi_{1,1}=0.8, \xi_{2,1}=0.25, \xi_{1,2}= 0.2, \xi_{2,2}=0.75$, $T=7.7$. }\label{solcost}
\end{figure}
\begin{figure}[htbp!]
\includegraphics[width=12cm,height=6.5cm]{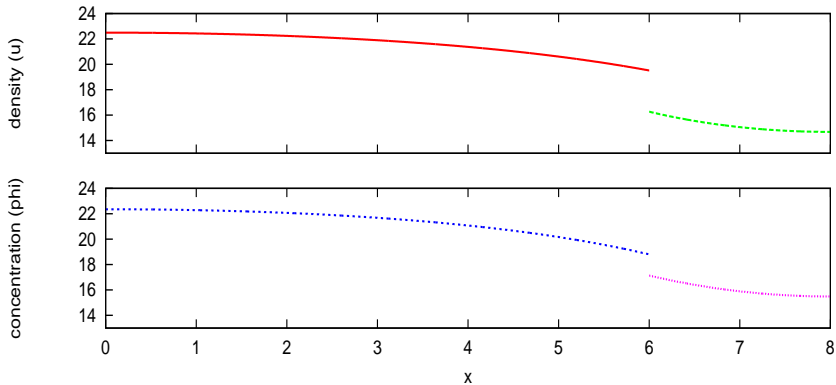}
\caption{Asymptotic solution at time $T=30$ for $\lambda_1=5$, $\lambda_2=4$, in case of non-dissipative coefficients $\xi_{1,1}=0.8, \xi_{2,1}=0.25, \xi_{1,2}= 0.24, \xi_{2,2}=0.7$.}\label{solnoncost}
\end{figure}

\begin{figure}[htbp!]
\begin{center}
\includegraphics[height=5cm,width=5cm]{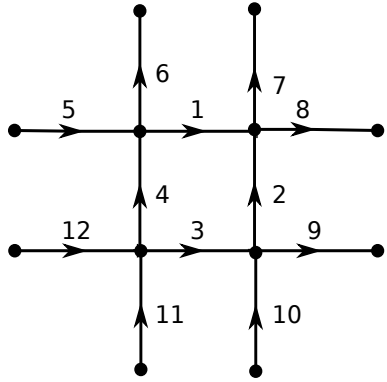}
\caption{A network composed of twelve arcs (six incoming and six outgoing)  connected by four internal nodes.}
\label{fig:12a4n}
\end{center}
\end{figure}

Let us now consider a larger network composed of twelve nodes and four arcs, see Fig. \ref{fig:12a4n}. We choose some non dissipative transmission coefficients, given in Table \ref{tableCoeff},  in order to satisfy condition (\ref{condition_lambda}). Let us consider as initial condition on the incoming arc 5,  the function plotted in Fig. \ref{indat5}, where we put a small symmetric perturbation of the constant state $u=110$.

\begin{table}[htb]
\centering
{
\begin{tabular}{|c|cccc|} \hline
&$\xi_{12,12}=0.1$,&
$\xi_{11,12}=0.3$,&
$\xi_{3,12}=0.3$,&
$\xi_{4,12}=0.3$, \\

 Node S-W &$\xi_{12,11}=0.2$,&
$\xi_{11,11}=0.2$,&
$\xi_{3,11}=0.3$,&
$\xi_{4,11}=0.3$, \\

&$\xi_{12,3}=0.2$,& 
$\xi_{11,3}=0.2$,&
$\xi_{3,3}=0.4$,&
$\xi_{4,3}=0.2$, \\

&$\xi_{12,4}=0.5$,&
$\xi_{11,4}=0.1$,&
$\xi_{3,4}=0.2$, &
$\xi_{4,4}=0.2$, \\  \hline

&$\xi_{3,3}=0.1$,&
$\xi_{10,3}=0.3$,&
$\xi_{9,3}=0.3$,&
$\xi_{2,3}=0.3$,\\

Node S-E &$\xi_{3,10}=0.2$,&
$\xi_{10,10}=0.2$,&
$\xi_{9,10}=0.3$,&
$\xi_{2,10}=0.3$,\\

&$\xi_{3,9}=0.2$,&
$\xi_{10,9}=0.2$,&
$\xi_{9,9}=0.4$,&
$\xi_{2,9}=0.2$,\\

&$\xi_{3,2}=0.5$,&
$\xi_{10,2}=0.1$,&
$\xi_{9,2}=0.2$,&
$\xi_{2,2}=0.2$,   \\ \hline

&$\xi_{1,1}=0.1$,&
$\xi_{2,1}=0.3$,&
$\xi_{8,1}=0.3$,&
$\xi_{7,1}=0.3$,\\

Node N-E &$\xi_{1,2}=0.2$,&
$\xi_{2,2}=0.2$,&
$\xi_{8,2}=0.3$,&
$\xi_{7,2}=0.3$,\\
  
&$\xi_{1,8}=0.2$,&
$\xi_{2,8}=0.2$,&
$\xi_{8,8}=0.4$,&
$\xi_{7,8}=0.2$,\\
 
&$\xi_{1,7}=0.5$,&
$\xi_{2,7}=0.1$,&
$\xi_{8,7}=0.2$,&
$\xi_{7,7}=0.2$,   \\ \hline

&$\xi_{5,5}=0.1$,&
$\xi_{4,5}=0.3$,&
$\xi_{1,5}=0.3$,&
$\xi_{6,5}=0.3$,\\

Node N-W &$\xi_{5,4}=0.2$,&
$\xi_{4,4}=0.2$,&
$\xi_{1,4}=0.3$,&
$\xi_{6,4}=0.3$,\\

&$\xi_{5,1}=0.2$,&
$\xi_{4,1}=0.2$,&
$\xi_{1,1}=0.4$,&
$\xi_{6,1}=0.2$,\\

&$\xi_{5,6}=0.5$,&
$\xi_{4,6}=0.1$,&
$\xi_{1,6}=0.2$,&
$\xi_{6,6}=0.2$.\\ \hline
\end{tabular}}
\vspace{0.1 in} 
\caption{Transmission coefficients used for the numerical simulations of Figures \ref{fig:12a4n_asympt} and \ref{fig:12a4n_flux} given node by node.}
\label{tableCoeff}
\end{table}

\begin{figure}
\includegraphics[width=5.5cm,height=5.5cm]{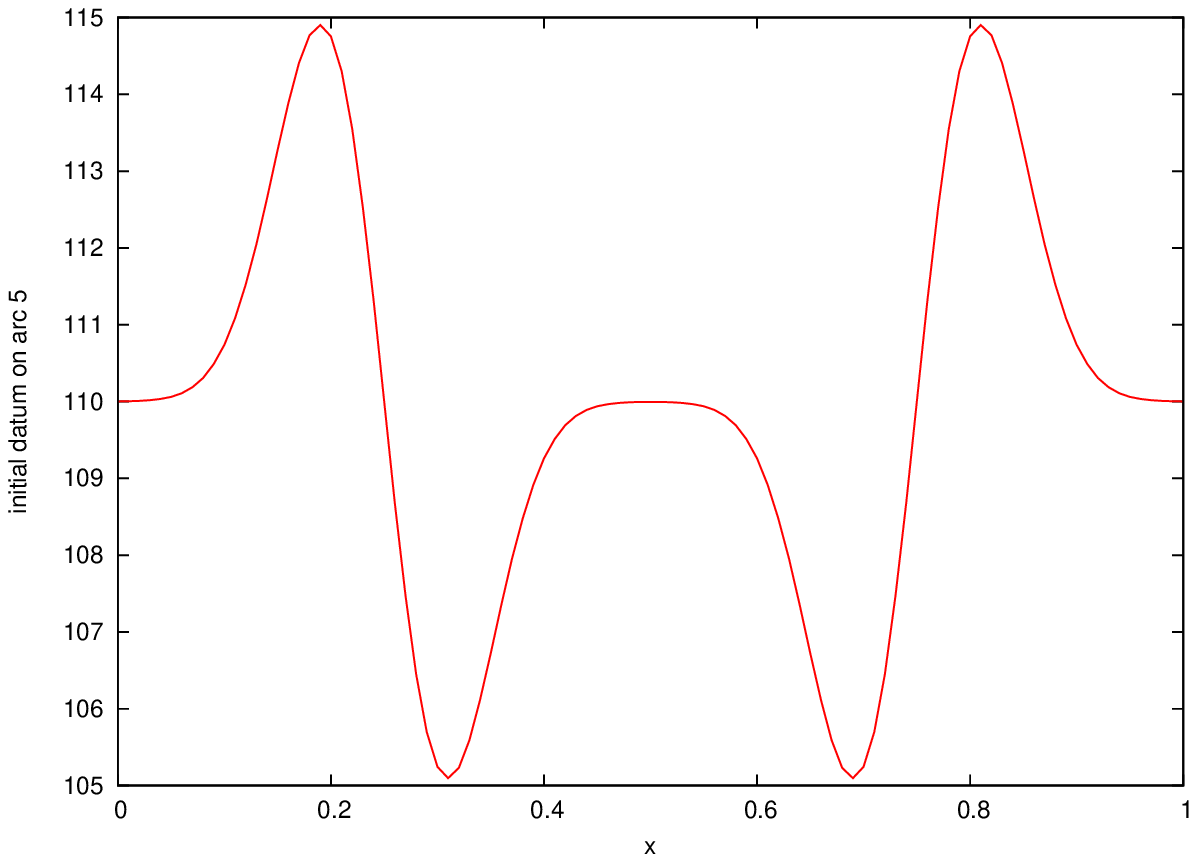}
\caption{Initial condition for $u$ and $\phi$ on arc 5 of the network presented in Fig. \ref{fig:12a4n}. }\label{indat5}
\end{figure}
In this case it is hard to compute analytically the stationary solutions. We only know that non-constant solutions are generally expected, according to the discussion in Section \ref{SS-staz}.
In Fig. \ref{fig:12a4n_asympt} we plot the asymptotic densities on the network node by node, starting from North-East and proceeding in a clockwise direction. Notice that most of the arcs are repeated in the different figures. In Fig. \ref{fig:12a4n_flux} the asymptotic fluxes are represented, and again our scheme is able to stabilize them correctly. We notice that the fluxes of arcs connected to outer boundaries vanish, whereas the fluxes of inner arcs, even if they are constant, are different from zero.

\begin{figure}[htbp!]
\centering
\includegraphics[height = 4.2 cm, width = 4.2 cm] {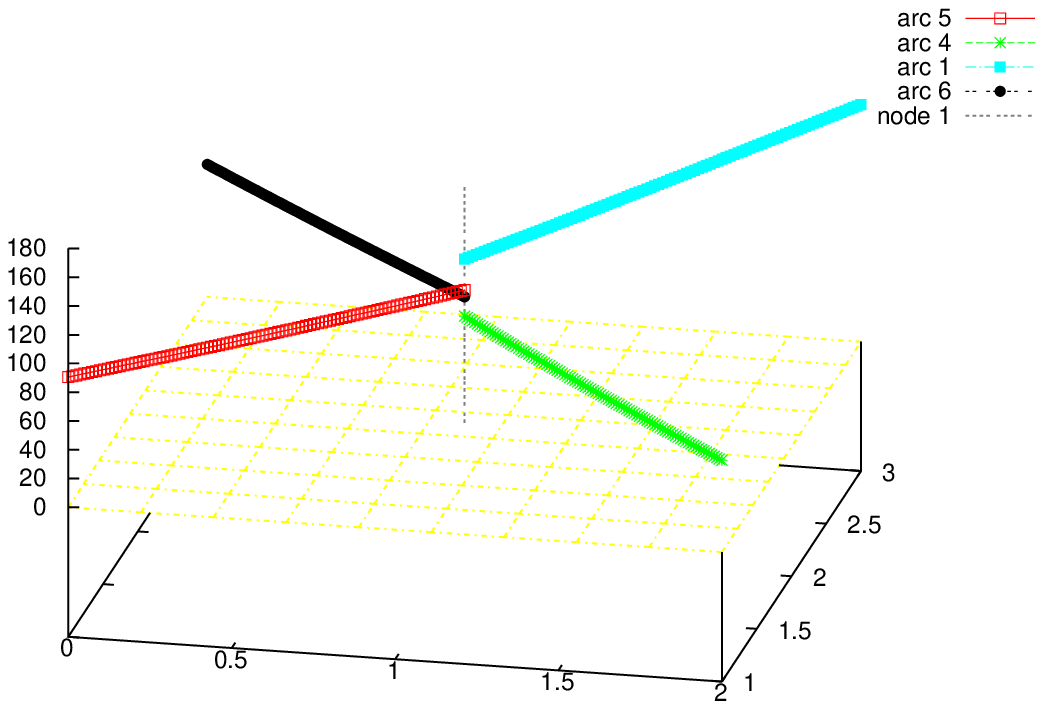}\qquad
\includegraphics[height = 4.2 cm, width = 4.2 cm] {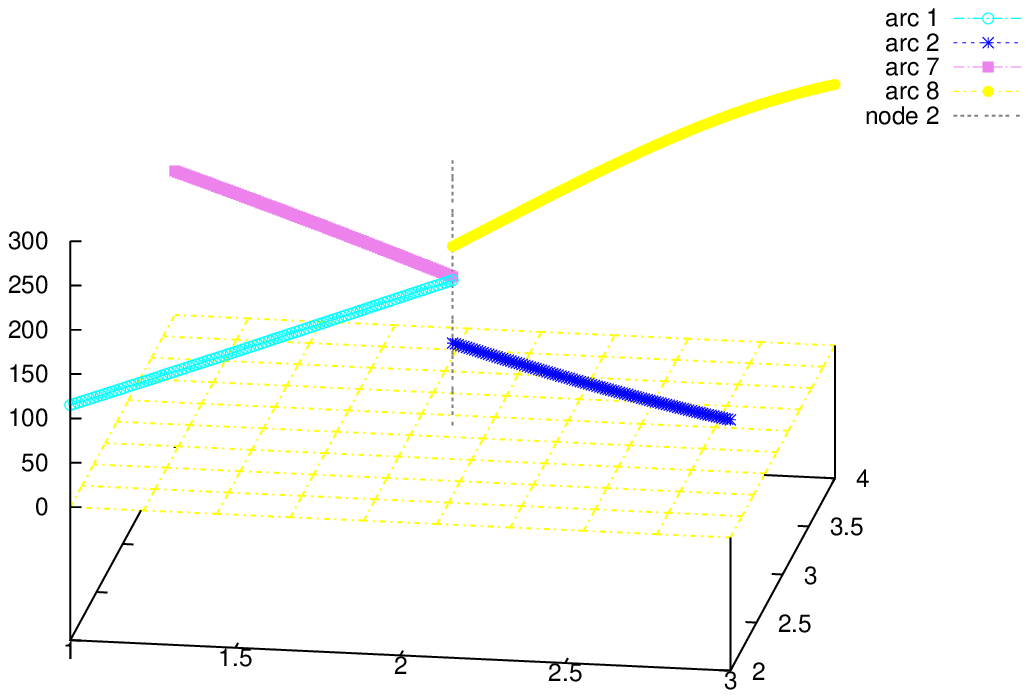}\\
\includegraphics[height = 4.2 cm, width = 4.2 cm] {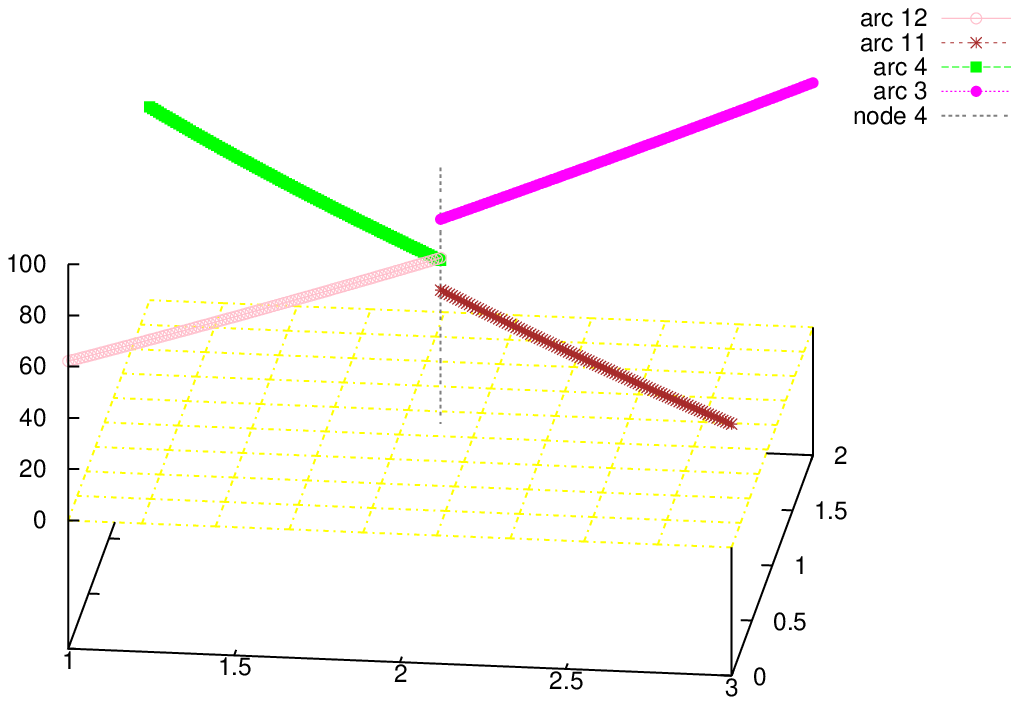}\quad
\includegraphics[height = 4.2 cm, width = 4.2 cm] {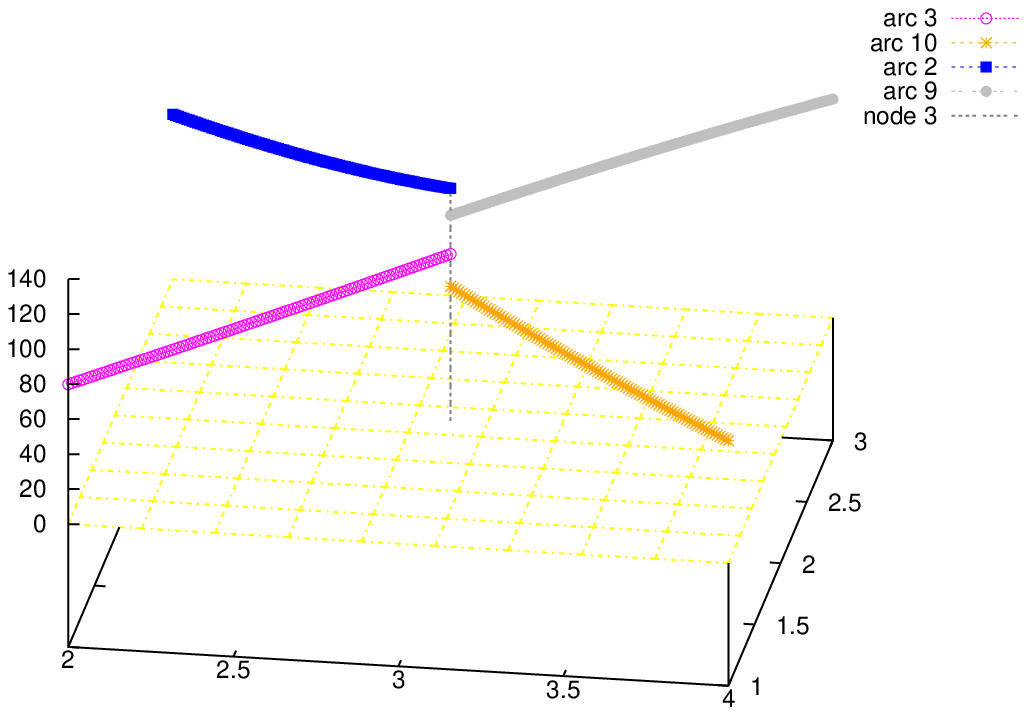}

\caption{Stationary solutions for the network composed of 12 arcs and 4 nodes of Fig. \ref{fig:12a4n}: the densities are computed at time $T= 30$, the values of the parameters are given by: $\lambda_i=\lambda=10, L_i=1, a_i=b_i=D_i=1$. The transmission coefficients can be found in Table \ref{tableCoeff}. The total initial mass $\mu_0= 1320$ is distributed as a perturbation of the constant state $110$ on arc 5 as in Fig. \ref{indat5} and as the constant density $110$ on the other arcs, with $h_i =h= 0.01, k=0.0005$.}\label{fig:12a4n_asympt}
\end{figure}

\begin{figure}[htbp!]
\begin{center}
\includegraphics[height=5.5cm,width=5.5cm]{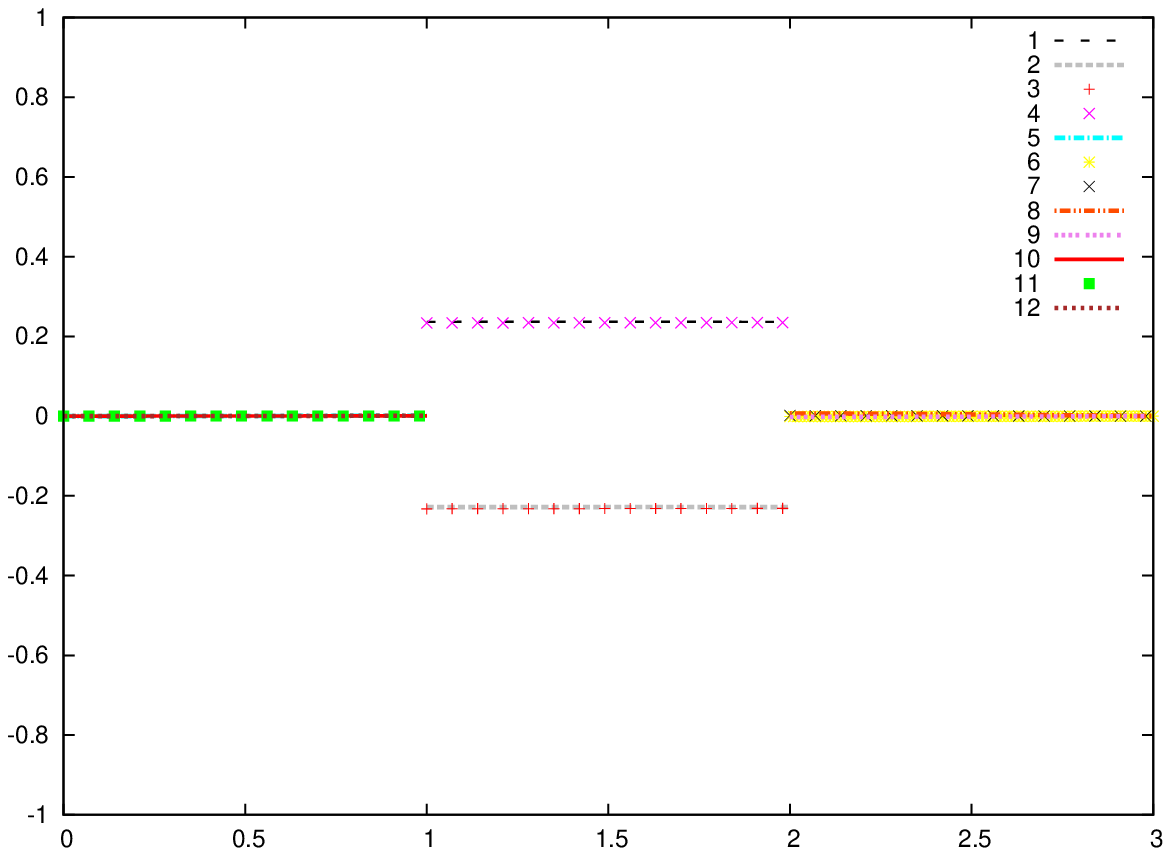}
\caption{The asymptotic fluxes of the arcs of the network composed of 12 arcs and 4 nodes at time $T=30$, with the same data as Fig. \ref{fig:12a4n_asympt}.}
\label{fig:12a4n_flux}
\end{center}
\end{figure}

\subsection{Instabilities: the appearance of numerical blow-up}

Let us consider some cases that present a strong asymptotical instability. Indeed, for some values of the parameters of the problem, namely of the arc's length $L$ and the cell velocity $\lambda$, in connection with the total mass  distributed on the arcs of the network, we can observe  increasing oscillations, which eventually may cause the blow-up of solutions.
It is important to notice that the blow-up can be already observed for this model even for a single arc, see Example 1 below, when the total mass $\mu_0$ is large  with respect to the characteristic parameters $L$ and $\lambda$. However, here the presence of more arcs, and so, a greater total length and total mass,   makes this kind of phenomenon much more frequent. 
 
\textbf{Example 1.}
Here we assume that we have only one interval with $L=1$ and $\lambda=10$ and we take, as initial condition for the density and the chemoattractant, a symmetric perturbation of a constant state  $C_0=9000$. The total mass is $\mu_0=9000$, as shown in  Figure \ref{datoex1}. The solution presents a clear blow-up at time $T=0.1$, see Fig. \ref{blowupex1}. This blow-up seems associated to non physical negative values of the density function $u$, and it is observed in the same way even for refined meshes (see Table 2 for the case of two arcs). This is not surprising, since the quasimonotonicity of the system, see again \cite{Na96}, is violated when the gradient $\phi_x$ is larger than $\lambda$.

\begin{figure}[htbp!]
\begin{center}
\includegraphics[height=5.5cm,width=5.5cm]{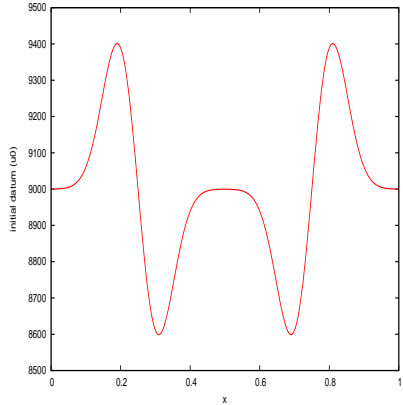}
\caption{The initial condition $u_0(x)$ is a symmetric perturbation of a constant state  $C_0=9000$, the total mass is $\mu_0=9000$.}
\label{datoex1}
\end{center}
\end{figure}

\begin{figure}[htbp!]
\begin{center}
\includegraphics[height=6cm,width=5.3cm]{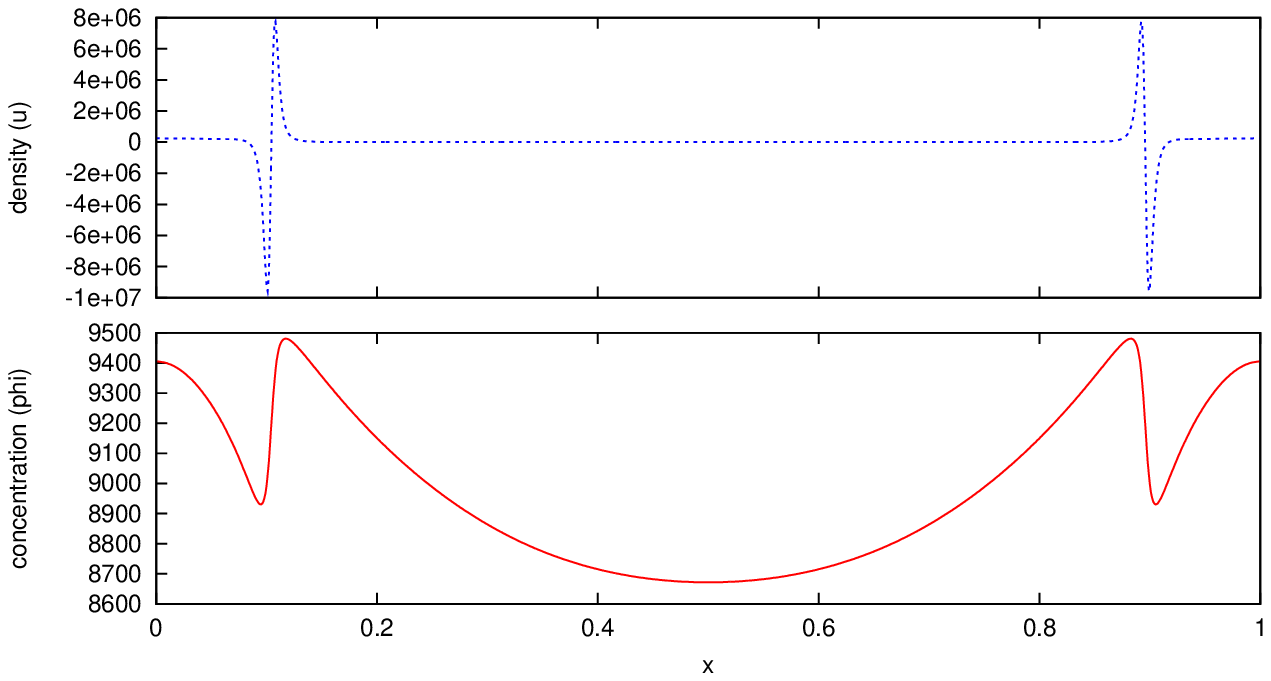}\quad \includegraphics[width=5cm,height=5.2cm]{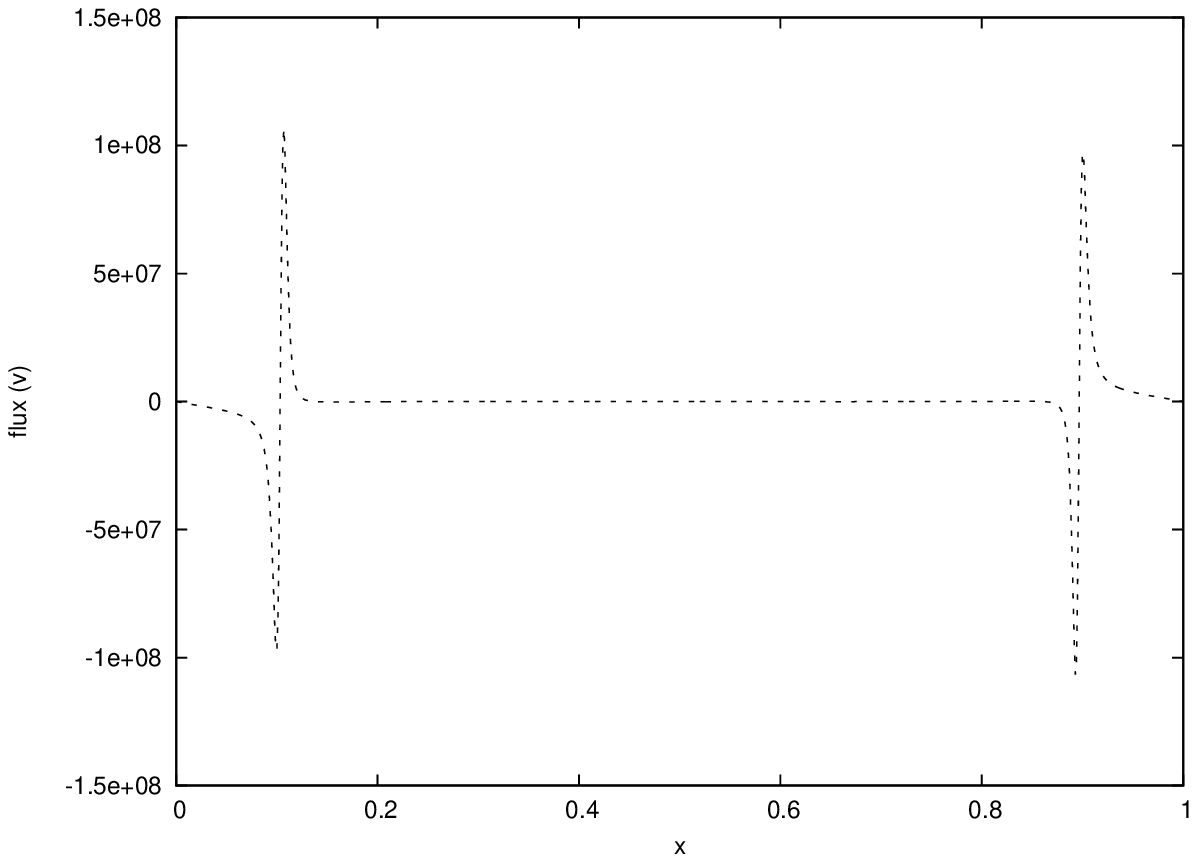}
\caption{Blow-up of the solution at time $T=0.1$, for data in Fig. \ref{datoex1} with $L=1$, $\lambda=10$, ${h} =0.001$, $\mu_0=9000$: on the left the blow-up density $u$ and the concentration $\phi$, on the right the flux $v$.}
\label{blowupex1}
\end{center}
\end{figure}


\textbf{Example 2.}
Here we take two arcs of length $L_1=6$ and $L_2=2$ and the initial density as in Fig. \ref{indat},  with $a_i=b_i=1$,  $u_i(x,0)=\phi_i(x,0)$, and $v_i(x,0)=0$, $i=1,2$. Then we change the values of velocities $\lambda_1$ and $\lambda_2$ in order to see how they  influence the behavior of solutions to system (\ref{hyper-gen}). At the junction we assume transmission and dissipative coefficients, taking  $\xi_{1,1}=0.96$ and then satisfying equations (\ref{dissip-coeff2})--(\ref{dissip-coeff22}).
What we observe  is that solutions blow up in finite time or not according to the relative values of $\lambda_1$ and $\lambda_2$, as it is shown in Figure \ref{regioni}. More precisely, we can observe three different regimes. If $\lambda_2$ is large with respect to $\frac{1}{\lambda_1-2}$, solutions stay bounded and converge to stationary solutions (green "x" in Figure \ref{regioni}). 
If $\lambda_1$ is small with $\lambda_2$ large enough, then solutions blow up in finite time (red "+" in Figure \ref{regioni}). Finally, there is a small region in between, $\lambda_1$ around the value 3 and $\lambda_2$ small enough, such that  solutions present a large  spike at the boundaries (marked by blue asterisks ``*'').
 
\begin{figure}
\begin{center}
\begin{psfrags} 
\psfrag{lambda1}{$\lambda_1$}\psfrag{lambda2}{$\lambda_2$}
\includegraphics[height=5cm]{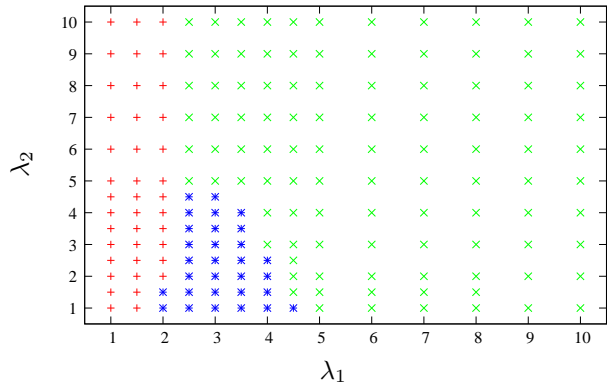} 
\end{psfrags} 
\caption{Schematization of the regions describing the behavior of solution for $\mu_0$ =160 and the velocities $\lambda_1$ and $\lambda_2$ varying: blow-up (marked by  red crosses ``+``), solutions with a spike at the boundaries (marked by blue asterisks ``*'') and stable stationary solutions (marked by  green ``x``).}\label{regioni}
\end{center}
\end{figure}
\vspace{0.2 in}

Let us now focus on the blow-up behavior. Referring to Fig. \ref{regioni}, we can choose a pair of velocities belonging to the blow-up region marked by red crosses ``+``, to say $\lambda_1=1$ and $\lambda_2=2$. The time step just before the numerical blow-up time of corresponding solutions, starting from initial data as in Fig. \ref{indat}, is plotted in Fig. \ref{blowupex2}. Even if apparently we are close to the transmission point, there are many grid points separating it from the blow-up point. To show that the blow-up is not just a numerical artifact, we perform the same simulation with the same data, but on refined grids. In Table \ref{table1} we report the blow-up time of solutions to system (\ref{hyper-gen}) for a fixed global mass $\mu_0$ when either the CFL condition $\nu = \frac{k}{h} \lambda$ or $h$ go to zero.
Out of the case of $\nu =1$, which appears to be more unstable, the blow-up time is independent of the  meshes and has to be considered to occur in the analytical solutions. 

\begin{figure} [htbp!]
\includegraphics[width=5.3cm,height=5.6cm]{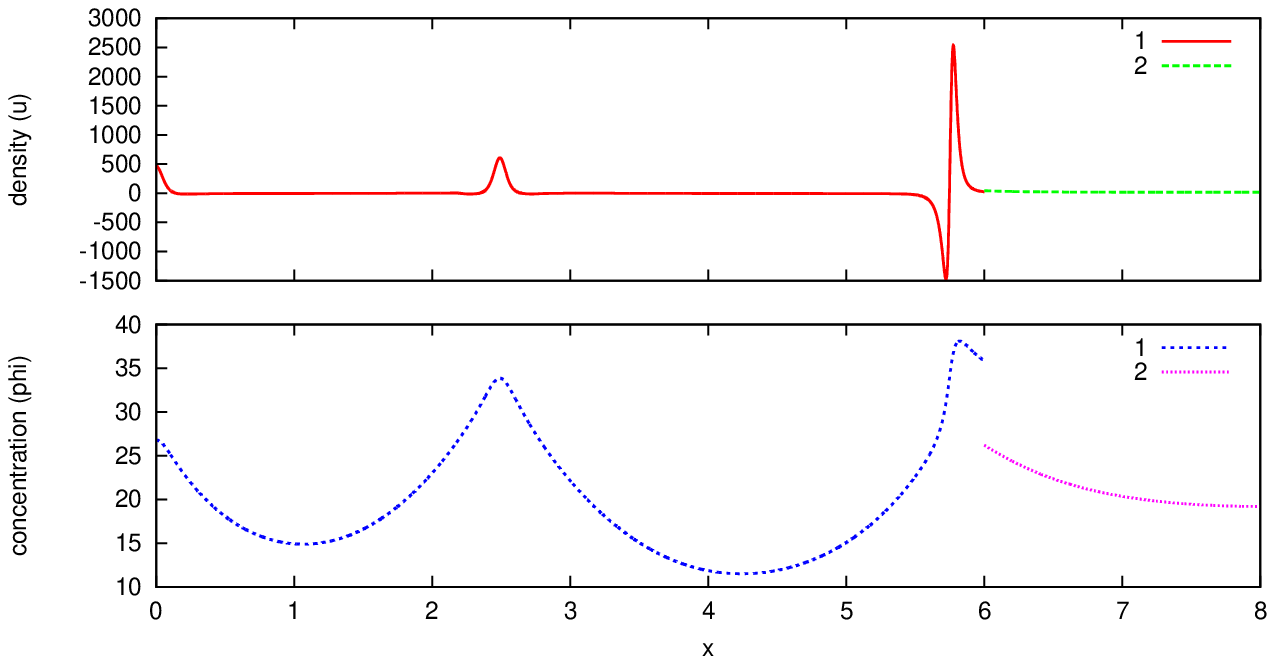}
\quad \includegraphics[width=5cm,height=5.5cm]{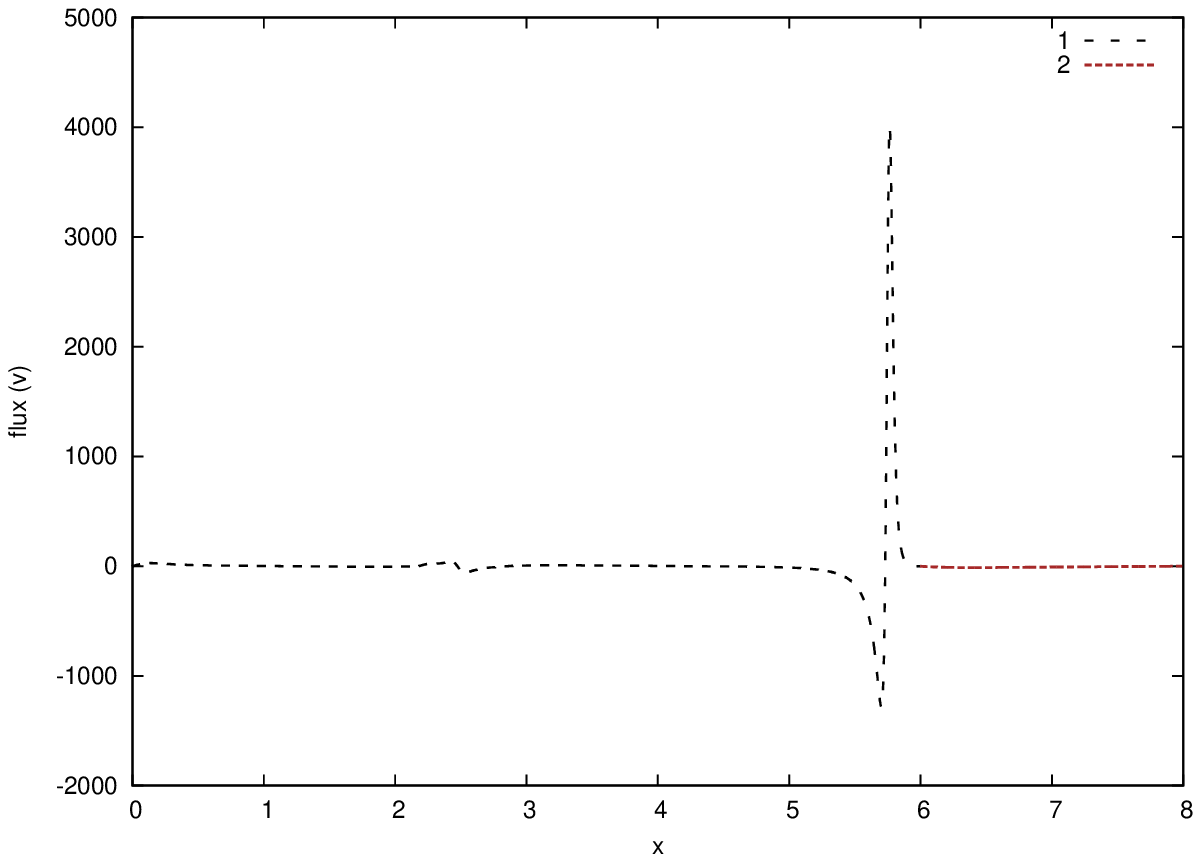}
\caption{Blow-up at time $T=4$, for initial  data as  in Fig. \ref{indat}, with $L_1=6, L_2=2$, $\lambda_1=1$ and $\lambda_2=2$, dissipative coefficients with $\xi_{1,1}=0.96$, the total mass is equal to $\mu_0=160$: on the left the density $u$ and the concentration $\phi$, on the right the flux $v$. The space steps are equal to ${h}_1= 0.001$, ${h}_2= 0.002$.}\label{blowupex2}
\end{figure}

\begin{table}[htbp!]
\centering
{
\renewcommand{\arraystretch}{1.2}
\begin{tabular}{|c|c|c|c|c|} \hline
\multicolumn{5}{|c|} {\rule[-3mm]{0mm}{8mm} Blow-up time} \\\hline\hline

$h$ & $\nu=1$ & $\nu=\frac{1}{2} \ \ \ \ $ & $\nu=\frac{1}{4} \ \ \ \ $ &$\nu=\frac{1}{8} \ \ \ \ $\\ \hline
$0.01$ &$2$  &  $4$ & $4$ & $4$\\
$0.0025$ & $1$  &  $4$ & $4$ & $4$ \\
$0.001$ & $0.5$  & $4$ &  $4$ & $4$ \\ \hline\hline

\end{tabular}}
\vspace{0.2 in}

\caption{Blow-up times of the solutions to system (\ref{hyper-gen}) when either the CFL condition $\nu = \frac{k}{h} \lambda$ or $h$ go to zero, with transmission coefficients of dissipative type, $L_1=6$, $L_2=2$, $\lambda_1=1$, $\lambda_2=2$, $\mu_0=160$.}
\label{table1}
\end{table}

\vspace{1cm}

\subsection{Comparisons and errors}

Let us now introduce the formal order of convergence of a numerical method $\gamma_w$ for the computation of the solution $w$ as the minimum among the orders on the arcs of the network:
\begin{equation}\label{order1}
\gamma_w = \min_{i}{\gamma_w^i},
\end{equation}
where
\begin{equation}\label{order2}
\gamma_w^i = \log_2\left(\frac{e^i({h}_i)}{e^i\left(\frac{{h}_i}{2}\right)}\right), \quad  i=1,\ldots,N.
\end{equation}
The $L^1$-error for the numerical solution on each arc is
\begin{equation}\label{error}
e^{i}\left(\frac{{h}_i}{n}\right) = \frac{{h}_i}{n} \sum_{l=0,..., nM_i}
  \left|w^T_l\left(\frac{{h}_i}{n}\right)
  -w^T_{2l}\left(\frac{{h}_i}{2n}\right)\right| \;\; n=1, 2,
\end{equation}
where $w^T_j(h)$ denotes the numerical solution obtained with the
space step discretization equal to $h$, computed in $x_j$ at the
final time $T$. The total $L^1$-error is
\begin{equation}
TOT_{err} = \sum_{i=1}^{N} e^{i}({h}_i).
\end{equation}

Table \ref{table2} shows the $L^1$-error (\ref{error}) on the asymptotic solutions $u$, $\phi$ and $v$ and order of convergence (\ref{order1}) of the approximation scheme applied to the considered network.
\begin{table}[htb]
\centering
{
\begin{tabular}{|c|c|c|c|c|c|c|} \hline
$h$ & $\gamma_u$& Error on $u$ & $\gamma_{\phi}$ &Error on $\phi$ &$\gamma_v$&Error on $v$\\ \hline
$0.025$&0.916393& 1.78849e-04 &0.965238 &1.78848e-04 &1.212334 &3.34559e-07\\ \hline
$0.0125$& 0.959614 &8.87206e-05 &0.982631 &8.87207e-05 &-0.058657 &1.44060e-07\\ \hline
$0.00625$&0.980243&4.41941e-05&0.990856&4.41954e-05&0.666605&1.49949e-07\\\hline
$0.003125$ &0.986317&2.20550e-05&0.992983&2.20651e-05&0.863690&9.43741e-08\\\hline
$0.0015625$ &0.937936&1.10172e-05&0.937109 &1.10280e-05&0.955806&5.17981e-08\\ \hline
\end{tabular}}
\vspace{0.1 in}
\caption{Orders and errors of the approximation scheme for the solutions to system (\ref{hyper-gen}), $L_i=1$, $\lambda_i=4, i=1,2$, $\mu_0=120.056$, $T=25$.}
\label{table2}
\end{table}

The results in Table \ref{table2} show the effectiveness of AHO approximation scheme in the solution of the transmission problem represented by the hyperbolic model (\ref{hyper-gen}). We notice indeed that even in this more general case the scheme still keeps a formal accuracy of first order, although  the interactions at the boundaries could deteriorate its accuracy.

\begin{acknowledgement}
The research leading to these results has received funding from the
European Union Seventh Framework Programme [FP7/2007-2013]  under
grant agreement n°257462 HYCON2 Network of excellence. This work has also been partially supported by the PRIN project 2008-2009 ``Equazioni iperboliche non lineari e fluidodinamica'' and by the ANR project  MONUMENTALG,  ANR-10-JCJC  0103.
\end{acknowledgement}


\end{document}